\documentclass{aic}

\usepackage{cite}
\usepackage{latexsym}
\usepackage{amsfonts}
\usepackage{multirow}

\newtheorem{dfn}{Definition}
\newtheorem{theorem}[dfn]{Theorem}
\newtheorem{lemma}[dfn]{Lemma}
\newtheorem{question}[dfn]{Question}

\newtheorem{corollary}[dfn]{Corollary}

\def\Bin{\mathrm{Bin}}

\newcommand{\E}{{\mathbb{E}}}

\newcommand{\lt}{\left}
\newcommand{\rt}{\right}

\newcommand{\Prob}{{\mathbb{P}}}

\aicAUTHORdetails{%
  title = {Bisections of graphs under degree constraints},
  author = {Jie Ma and Hehui Wu},
  plaintextauthor = {Jie Ma, Hehui Wu},
  plaintexttitle = {Bisections of graphs under degree constraints},
  runningtitle = {Bisections of graphs under degree constraints},
  runningauthor = {Jie Ma and Hehui Wu},
  copyrightauthor = {Jie Ma and Hehui Wu},
}

 \aicEDITORdetails{%
    year={2026},
   number={5},
    received={24 June 2025},
    published={17 July 2026},
    doi={10.19086/aic.2026.5},
 }

\begin{document}

\begin{frontmatter}
\title{Bisections of graphs under degree constraints}
\date{}

\author[jie]{Jie Ma\thanks{Research supported by National Key Research and Development Program of China 2023YFA1010201 and National Natural Science Foundation of China grant 12125106.}}
\author[hehui]{Hehui Wu\thanks{Research supported by National Natural Science Foundation of China grants 12371343 and 12525110, National Key Research and Development Program of China 2020YFA0713200, and the Shanghai Dawn Scholar Program grant 19SG01.}}

\begin{abstract}
		In this paper, we investigate the problem of finding {\it bisections} (i.e., balanced bipartitions) in graphs.
		We prove the following two results for {\it all} graphs $G$:
		\begin{itemize}
			\item $G$ has a bisection where each vertex $v$ has at least $(1/4 -o(1))d_G(v)$ neighbors in its own part;
			\item $G$ also has a bisection where each vertex $v$ has at least $(1/4 -o(1))d_G(v)$ neighbors in the opposite part.
		\end{itemize}
		These results are asymptotically optimal up to a factor of $1/2$, aligning with what is expected from random constructions,
		and provide the first systematic understanding of bisections in general graphs under degree constraints.
		As a consequence, we establish for the first time the existence of a function $f(k)$ such that for any $k\geq 1$, every graph with minimum degree at least $f(k)$ admits a bisection where every vertex has at least $k$ neighbors in its own part,
		as well as a bisection where every vertex has at least $k$ neighbors in the opposite part.
		
		Using a more general setting, we further show that for any $\varepsilon > 0$, there exist $c_\varepsilon, c'_\varepsilon > 0$ such that any graph $G$ with minimum degree at least $c_\varepsilon k$ (respectively, $c'_\varepsilon k$) admits a bisection satisfying:
		\begin{itemize}
			\item every vertex has at least $k$ neighbors in its own part (respectively, in the opposite part), and
			\item at least $(1 - \varepsilon)|V(G)|$ vertices have at least $k$ neighbors in the opposite part (respectively, in their own part).
		\end{itemize}
		These results extend and strengthen classical graph partitioning theorems of Erd\H{o}s, Thomassen, and K\"{u}hn--Osthus, while additionally satisfying the bisection requirement.

\end{abstract}
\end{frontmatter}

\section{Introduction}
	\noindent
	The problem of graph partitioning (i.e., partitioning the vertex set of a graph into two or more parts) has been a central topic in combinatorics and theoretical computer science since the advent of modern graph theory in the last century.
	It has motivated a substantial body of research, spanning classical results in structural and extremal graph theory (see, e.g., \cite{Lov,Ed73,Ed75}) to major advances in computational complexity and approximation algorithms for problems such as Max-Cut (see \cite{Karp,GW}), among others.
	While traditional results in graph partitioning often optimize a single quantity (such as \cite{Ed73,Ed75,GW}),
	recent work increasingly focuses on partitioning under multiple simultaneous constraints.
	
	A natural problem is to partition graphs under {\it degree constraints}.
	In this direction, an early result by Erd\H{o}s shows that every graph with minimum degree at least $2k-1$ admits a bipartition such that
	\begin{equation}\label{equ:maxcut}
		\mbox{every vertex has at least $k$ neighbors in the opposite part.}
	\end{equation}
	Complementarily, Thomassen \cite{T83} established the existence of a function $f(k)$ such that every graph $G$ with minimum degree at least $f(k)$ admits a bipartition $V(G)=A\cup B$ satisfying
	\begin{equation}\label{equ:Thom}
		\mbox{every vertex has at least $k$ neighbors in its own part,}
	\end{equation}
	i.e., both induced subgraphs $G[A]$ and $G[B]$ have minimum degree at least $k$.
	This function $f(k)$ was improved by Hajnal \cite{Haj} and was later determined to be $f(k)=2k+1$ by Stiebitz \cite{S96}, resolving a conjecture of Thomassen~\cite{T88}.
	A compelling question arises regarding whether both \eqref{equ:maxcut} and \eqref{equ:Thom} can hold simultaneously for graphs with sufficiently large minimum degree.
	K\"{u}hn and Osthus~\cite{KO} constructed examples showing this is impossible even when $k=1$.
	Despite this limitation, they proved a remarkable compromise:
	there exists a function $g(k)$ such that every graph $G$ with minimum degree at least $g(k)$ admits a bipartition $V(G)=A\cup B$ such that
	\begin{equation}\label{equ:KO}
		\mbox{\eqref{equ:Thom} holds, and every vertex in $A$ has at least $k$ neighbors in $B$.}
	\end{equation}
	The function $g(k)$ was recently improved to be a linear function in \cite{MW}, answering a question from~\cite{KO}.
	In a different direction, K\"uhn and Osthus \cite{KO} generalized Thomassen's result by further bounding the average degree of the cut across the bipartition from below:
	every graph $G$ with minimum degree at least $2^{32}k$ admits a bipartition $V(G)=A\cup B$ such that
	\begin{equation}\label{equ:KO2}
		\mbox{\eqref{equ:Thom} holds, and the bipartite subgraph $(A,B)_G$ has average degree at least $k$.}\footnote{In additional, the sizes of the parts $A$ and $B$ are at least $|V(G)|/2^{18}$; see Theorem~5 in \cite{KO}.}
	\end{equation}
	The results in \eqref{equ:KO} and \eqref{equ:KO2} reveal novel aspects of graph partitioning, which also lead to important applications in structural graph theory (see \cite{KO}).

	Another natural problem that has garnered considerable attention in recent decades is graph partitioning with {\it part-size constraints}.
	A canonical and notable example is the search for a {\it bisection} in graphs $G$, that is, a bipartition $V(G) = A \cup B$ satisfying $\big||A| - |B|\big| \leq 1$.
	Extensive research has been conducted on finding bisections with various properties; see the surveys \cite{BS,Scott} and references \cite{DPS,FHY,HW21,HWY,LLS,LZ,LMZ,XY}, most of which focus on bounding the number of edges (or arcs in digraphs) either between the two parts or within each part.
	A prominent conjecture of Bollob\'as and Scott \cite{BS} (see its Conjecture~8) in this area states that every graph $G$ has a bisection satisfying \begin{equation}\label{conj:BS}
		\mbox{every vertex $v$ has at least $\frac{d_G(v)}{2}-\frac{1}{2}$ neighbors in its opposite part.}
	\end{equation}
	This original statement fails for an infinite family of graphs, as shown in \cite{JMYY}, but its equally important weaker version, which replaces the additive constant $-\frac{1}{2}$ with $O(1)$ in \eqref{conj:BS}, remains widely open.
	In \cite{BL}, Ban and Linial investigated several partitioning problems for regular graphs and proposed a related conjecture (see Conjecture~1 in \cite{BL}), asserting that every bridgeless cubic graph, except for the Petersen graph, has a bisection such that every vertex has at least two neighbors in its opposite part.
	Another longstanding conjecture of F\"uredi, which also appears in Green's 100 open problem list \cite{Green},
	states that with high probability, the Erd\H{o}s-R\'enyi random graph $G(n,1/2)$ has a bisection in which all but $o(n)$ vertices have at least half of their neighbors in the opposite part.
	This was recently proved by Ferber, Kwan, Narayanan, Sah, and Sawhney \cite{FKNSS}, who also showed that with high probability, $G(n,1/2)$ has a bisection in which all but $o(n)$ vertices have at least half of their neighbors in their own part.
	More recently, further strengthening and advances related to this conjecture on bisections in random graphs were achieved by Dandi, Gamarnik, and Zdeborov\'a \cite{DGZ}, Minzer, Sah, and Sawhney \cite{MSS}, as well as Anastos, Cooley, Kang, and Kwan \cite{ACKK}.

	In this paper, we investigate bisections in general graphs under degree constraints.
	We develop a unified approach that extends to multi-partitions with arbitrary part-size constraints.
	Our results also reveal new partitioning phenomena that both generalize and strengthen the classical theorems of Erd\H{o}s, Thomassen, and K\"{u}hn and Osthus discussed above.

	\subsection{Bisections in general settings}
	\noindent
	We first present our main results on bisections in the most general settings.
	In particular, the external-degree statement below (Theorem~\ref{Thm:ext-bisec}) may be viewed as a first general step toward the Bollob\'as--Scott conjecture \eqref{conj:BS},
	since it guarantees for every graph a bisection in which each vertex sends a positive constant fraction of its incident edges across the cut.
	The following two results hold for all graphs.
	
	\begin{theorem}\label{Thm:int-bisec}
		Every graph $G$ has a bisection such that
		every vertex $v$ has at least $$\frac{d_G(v)}{4}-o(d_G(v)) \mbox{ ~ neighbors in its own part.}$$
	\end{theorem}
	
	\begin{theorem}\label{Thm:ext-bisec}
		Every graph $G$ has a bisection such that every vertex $v$ has at least
		$$\frac{d_G(v)}{4}-o(d_G(v)) \mbox{ ~ neighbors in its opposite part.}$$
	\end{theorem}
	
	From the perspective of random bisections, we observe that these results are asymptotically optimal up to a factor of $1/2$.
	More broadly, Theorems~\ref{Thm:int-bisec} and \ref{Thm:ext-bisec} are the first results concerning bisections in general graphs where the degree of every vertex is at least a constant fraction of its total degree.
	As a consequence, we establish the existence of the following function $f(k)$, which provides the bisection analogs of \eqref{equ:maxcut} and \eqref{equ:Thom}.
	
	\begin{corollary}\label{coro:main}
		There exists a function $f(k)=O(k)$ such that every graph with minimum degree at least $f(k)$ admits a bisection where every vertex has at least $k$ neighbors in its own part, as well as a bisection where every vertex has at least $k$ neighbors in the opposite part.
	\end{corollary}
	
	\subsection{Partitions with internal degree constraints}\label{subsec:intro-int}
	\noindent
	Next, we consider partitioning graphs with high minimum degree by imposing {\it internal degree constraints}, that is, by restricting the degrees of vertices within their own parts.
	The following theorem (derived from a more technical setting - Theorem~\ref{Thm:min-indegree}) guarantees the existence of a specific tripartition under part-size and internal degree constraints.

	\begin{theorem}\label{Thm:int-bisec-exact}
		Let $c, \varepsilon$ be constants satisfying $0\leq c<1$ and $0<\varepsilon\leq 1-c.$
		Then there exists $k_0(c, \varepsilon)$ such that for every integer $k\geq k_0(c, \varepsilon)$,
		every graph $G$ with minimum degree at least $\left(\frac{4}{1-c}+\varepsilon\right)k$ admits a tripartition $V(G)=A\cup B\cup C$ such that
		\begin{itemize}
			\item[(1).] $\left(\frac{1-c-\varepsilon}{2}\right)n\leq |A|,|B|\leq \left(\frac{1-c}{2}\right)n$,
			\item[(2).] $d_{A}(x)\geq k$ holds for every vertex $x\in A$,
			\item[(3).] $d_{B}(y)\geq k$ holds for every vertex $y\in B$, and
			\item[(4).] both $d_{A}(z)\geq 2k$ and $d_{B}(z)\geq 2k$ hold for every vertex $z\in C$,
		\end{itemize}
	\end{theorem}

	Using this theorem, we can immediately derive corresponding results on bisections.
	In the coming corollary, we apply Theorem~\ref{Thm:int-bisec-exact} with $1 - c = O(\varepsilon)$ and distribute the vertices of $C$ arbitrarily between $A$ and $B$ to form a bisection.
	
	\begin{corollary}\label{coro:int-bisec-main}
		For any positive constant $\varepsilon$, there exists a constant $c_\varepsilon>0$ such that the following holds.
		For any $k\geq 1$, every graph $G$ with minimum degree at least $c_\varepsilon\cdot k$ admits a bisection such that
		\begin{itemize}
			\item[(i).] every vertex has at least $k$ neighbors in its own part, and
			\item[(ii).] at least $(1-\varepsilon)|V(G)|$ vertices in $G$ have at least $k$ neighbors in the opposite part.
		\end{itemize}
	\end{corollary}
	
	We previously noted that it is impossible to ensure that all vertices satisfy both \eqref{equ:maxcut} and \eqref{equ:Thom} simultaneously, even for graphs $G$ with arbitrarily large minimum degree (as shown by counterexamples in \cite{KO}).
	However, this corollary indicates that a slightly weaker statement holds true: there exists a bisection in $G$ such that all vertices satisfy \eqref{equ:Thom}, while all but $o(|V(G)|)$ vertices satisfy \eqref{equ:maxcut}.
	Additionally, we observe that condition (ii) automatically provides a lower bound on the average degree of the bipartite subgraph induced by the bisection.
	Thus, this corollary directly extends the result \eqref{equ:KO2} of K\"uhn and Osthus \cite{KO}.
	
	In the next corollary, we establish explicit linear constants for the minimum degree conditions that guarantee bisections with desired internal degrees as well as average degree of the cut.
	In particular, this provides quantitative improvements to the aforementioned result of K\"uhn and Osthus \cite{KO}.
	
	\begin{corollary}\label{coro:int-bisec}
		For any positive constant $\varepsilon$, there exists $k_0(\varepsilon)$ such that the following holds for all integers $k \ge k_0(\varepsilon)$.
		Every graph $G$ with minimum degree at least $(4 + \varepsilon)k$ admits a bisection $V(G) = A \cup B$ such that
		both $G[A]$ and $G[B]$ have minimum degree at least $k$.
		Moreover, if the minimum degree of $G$ is at least $\left(\frac{16}{3} + \varepsilon\right)k$, then in addition to the above, the bipartite subgraph $(A, B)_G$ has average degree at least $k$.
	\end{corollary}
	
	\subsection{Partitions with external degree constraints}\label{subsec:intro-ext}
	\noindent
	Now we consider partitioning graphs with high minimum degree by imposing {\it external degree constraints},
	meaning we restrict the degrees of vertices within their opposite parts.
	The following is the main result of this subsection.
	
	\begin{theorem}\label{Thm:ext-bisec-exact}
		Let $c, \varepsilon$ be constants satisfying $0\leq c<1$ and $0<\varepsilon\leq 1-c.$
		Then there exists $k_0(c, \varepsilon)$ such that for every integer $k\geq k_0(c, \varepsilon)$,
		every graph $G$ with minimum degree at least $\left(\frac{4}{1-c}+\varepsilon\right)k$ admits a tripartition $V(G)=A\cup B\cup C$ such that
		\begin{itemize}
			\item[(1).] $\left(\frac{1-c-\varepsilon}{2}\right)n\leq |A|,|B|\leq \left(\frac{1-c}{2}\right)n$,
			\item[(2).] $d_{(A,B)}(v)\geq k$ holds for every vertex $v\in A\cup B$, and
			\item[(3).] both $d_A(z)\geq 2k$ and $d_B(z)\geq 2k$ hold for every vertex $z\in C$.
		\end{itemize}
	\end{theorem}
	
	Using this, we can derive the following corollaries in a manner similar to the previous subsection concerning internal degree constraints.
	
	\begin{corollary}\label{coro:ext-bisec-main}
		For any positive constant $\varepsilon$, there exists a constant $c_\varepsilon>0$ such that the following holds.
		For any $k\geq 1$, every graph $G$ with minimum degree at least $c_\varepsilon\cdot k$ admits a bisection such that
		\begin{itemize}
			\item[(i).] every vertex has at least $k$ neighbors in the opposite part, and
			\item[(ii).] at least $(1-\varepsilon)|V(G)|$ vertices in $G$ have at least $k$ neighbors in their own part.
		\end{itemize}
	\end{corollary}
	
	This offers a complementary statement to Corollary~\ref{coro:int-bisec-main}:
	every graph $G$ with large minimum degree has a bisection such that all vertices satisfy \eqref{equ:maxcut}, while all but $o(|V(G)|)$ vertices satisfy \eqref{equ:Thom}.
	
	\begin{corollary}\label{coro:ext-bisec}
		For any positive constant $\varepsilon$, there exists $k_0(\varepsilon)$ such that the following holds for all integers $k \ge k_0(\varepsilon)$.
		Every graph $G$ with minimum degree at least $(4 + \varepsilon)k$ admits a bisection $V(G) = A \cup B$ such that
		the bipartite subgraph $(A,B)_G$ has minimum degree at least $k$.
        Moreover, if the minimum degree of $G$ is at least $\left(\frac{16}{3} + \varepsilon\right)k$, then in addition to the above, both $G[A]$ and $G[B]$ have average degree at least $k$.
	\end{corollary}

	We suspect that both the $(4+\varepsilon)k$ bounds in Corollaries \ref{coro:int-bisec} and \ref{coro:ext-bisec} are within a factor of 2 of being optimal (see the discussion in Section~\ref{sec:conclusion}).
	
	The partition results with external degree constraints may appear analogous to those with internal constraints in Subsection~\ref{subsec:intro-int}. Indeed, as in the internal case, their proofs share the same two-stage framework: an initial random partitioning step followed by a deterministic refinement step. However, we emphasize that the deterministic refinement step requires substantially different approaches for these two settings.
	For readers interested in the technical details, we provide proof outlines before the main lemmas in Sections~\ref{sec:int} and \ref{sec:ext} (for internal and external degree constraints, respectively), where both the shared framework and the key differences are discussed.
	
	\bigskip
	
	\noindent{\bf Organization.}
	The remainder of this paper is organized as follows.
	In Section~\ref{sec:pre}, we present the preliminaries, including definitions, probabilistic inequalities,
	and an important lemma for finding certain dense subgraphs that is crucial for both settings of degree constraints.
	In Section~\ref{sec:int}, we prove our main results on graph partitions and bisections with internal degree constraints.
	The analogous results for external degree constraints are established in Section~\ref{sec:ext}.
	We conclude with a discussion of several remarks and open problems in the final section.
	
	\medskip
	
	\noindent{\bf Proof overview.}
	Our proofs combine a random starting partition with a deterministic cleaning procedure.
	We first choose a random tripartition with prescribed expected part sizes.
	For most vertices, Chernoff-type estimates show that the desired degree inequalities already hold with some slack.
	Lemma~\ref{Lem:exist C} then bounds the weighted number of vertices for which this random step fails, and this is exactly the quantity that controls all later vertex relocations.
	Next, Lemma~\ref{Lem:key} is used to identify a large robust set of vertices whose degree condition is already permanent.
	The deterministic stage then repeatedly moves a bad vertex, together with a controlled set of its neighbors, while the weighted estimate from Lemma~\ref{Lem:exist C} guarantees that only a small proportion of vertices are affected.
	A final redistribution of the remaining exceptional vertices yields the required tripartition and hence the desired bisection.

	\section{Preliminaries}\label{sec:pre}
	\noindent In this section, we gather and prepare all necessary preliminaries for the proofs that follow.
	
	\subsection{Basic notions}
	\noindent
	We adopt standard graph theory terminology. Let $G$ be a finite simple graph.
	For each vertex $v \in V(G)$, its \emph{neighborhood} $N_G(v)$ consists of all vertices adjacent to $v$,
	and its \emph{degree} is $d_G(v) = |N_G(v)|$. For any subset $A \subseteq V(G)$,
	we write $d_A(v) = |N_G(v) \cap A|$ for the number of neighbors of $v$ in $A$.
	For any integer $r \geq 2$, an {\it $r$-partition} of $G$ is a partition $V(G) = V_1 \cup V_2 \cup \cdots \cup V_r$ into $r$ non-empty parts.
	When $r=2$, we refer to this as a {\it bipartition}.
	Given a bipartition $V(G) = A \cup B$, the notation $(A,B)_G$ denotes both the edge cut
	(i.e., the set of all edges between $A$ and $B$) and the corresponding bipartite subgraph of $G$.
	When it is clear from the context, we often omit the subscript $G$ from the above notations.
	A {\it bisection} of $G$ is a bipartition $V(G) = A \cup B$ satisfying $\big| |A|-|B|\big|\leq 1$.
	For integers $n\geq \ell\geq 1$, we write $[n]=\{1,2,...,n\}$ and $\binom{[n]}{\ell}=\{F\subseteq [n]: |F|=\ell\}$.
	

	\subsection{Probabilistic inequalities}
	\noindent
	In this subsection, we collect several inequalities required for estimating graph-theoretic parameters during the initial random partitioning stage of our main results.
	First, we recall {\it Markov's inequality}.
	
	\begin{lemma}[see \cite{AS}]\label{Lem:Markov}
		Let $X$ be a nonnegative random variable and $\lambda > 0$. Then $\Prob(X\ge \lambda)\leq \E[X]/\lambda$.
	\end{lemma}
	
	We will also require the following {\it Chernoff bound} for binomial distributions. Let $\exp(x)=e^x$.
	
	\begin{lemma}[see \cite{AS}]\label{Lem:Chernoff}
		Let $p\in (0,1)$ and $X\sim \Bin(N,p)$. For any $\lambda > 0$, it holds that
		$$\Prob(|X-pN|>\lambda pN) \leq 2 \exp(-\lambda^2pN/3).$$
	\end{lemma}

	The following inequality bounds a key quantitative measure in our analysis.
	Intuitively, it controls how many vertices and edges can be affected when ``trouble'' vertices - those exhibiting large degree deviations from the initial random bipartition - are relocated during subsequent deterministic operations.
	Recall the {\it gamma function} $\Gamma(t)=\int_{0}^{+\infty} x^{t-1} e^{-x} dx$.

	\begin{lemma}\label{Lem:exist C}
		For any $\varepsilon\in (0,1)$ and any $\rho\in (0,1]$, by setting
		$d_{\varepsilon,\rho}:=\frac{1000}{\sqrt{\rho}\cdot \varepsilon^2},$
		we have
		$$\sum_{i=1}^{+\infty} i\cdot \exp(-d_{\varepsilon,\rho}^2i^{\varepsilon})\leq \frac{\rho\varepsilon^2}{10^5}.$$
		In particular, one may take $d_\varepsilon:=d_{\varepsilon,1}=1000/\varepsilon^2$.
	\end{lemma}
	
	\begin{proof}
		Fix $\varepsilon\in (0,1)$ and $\rho\in (0,1]$.
		The proof requires the following upper bound for the gamma function:
		\begin{equation}\label{equ:Gamma}
			\Gamma(t)\leq 3t^{t-1/2} e^{-t}\leq 3 t^t \mbox{ for } t\geq 2,
		\end{equation}
		where the first inequality can be derived from Stirling's formula.
		Using \eqref{equ:Gamma}, we obtain
		$$\sum_{i=1}^{+\infty} i\cdot \exp(-d_{\varepsilon,\rho}^2i^{\varepsilon})\leq \int_{0}^{+\infty} x\cdot \exp(-d_{\varepsilon,\rho}^2 x^{\varepsilon})dx=
		\frac{\Gamma(2/\varepsilon)}{\varepsilon\cdot d_{\varepsilon,\rho}^{4/\varepsilon}}\leq \frac{3}{\varepsilon}\cdot \left(\frac{2}{d_{\varepsilon,\rho}^2\varepsilon}\right)^{2/\varepsilon}
		\leq \frac{6}{d_{\varepsilon,\rho}^2\varepsilon^2}\leq \frac{\rho\varepsilon^2}{10^5},$$
		where the equality is obtained by substituting $d_{\varepsilon,\rho}^2 x^{\varepsilon}$ as a new variable,
		the second last inequality uses the fact $2/\varepsilon\geq 1$,
		and the final inequality follows from the choice $d_{\varepsilon,\rho}=1000/(\sqrt{\rho}\cdot \varepsilon^2)$.
		This proves the lemma.
	\end{proof}

	\subsection{Key lemma for finding dense subgraphs}
	\noindent We now introduce a key lemma that guarantees the existence of an induced subgraph with all vertices satisfying the prescribed degree conditions in a general graph.
	At its core, this result generalizes the classical result that every graph with average degree at least $2d$ contains a non-empty induced subgraph of minimum degree at least $d$.\footnote{It is worth noting that this bound is tight, as demonstrated by the complete bipartite graph $K_{d,n}$.}
	We emphasize that this lemma presents the decisive technical barrier that prevents improvements to the halfway bounds
	in our main theorems (Theorems~\ref{Thm:int-bisec} and \ref{Thm:ext-bisec}).
	
	\begin{lemma}\label{Lem:key}
		Let $H$ be a graph with a family $\{A_i\}_{i \in I}$ of disjoint vertex subsets,\footnote{Here, $\bigcup_{i \in I} A_i$ may not equal $V(H)$.} where $I$ is a finite index set.
		For each $i \in I$, let $\eta_i > 0$ be a real number and let $a_i \geq 1$ be an integer. Define the subset
		\[
		A_i^+ = \{v \in A_i : d_H(v) \geq 2(1+\eta_i)a_i\}.
		\]
		and set $\eta = \min_{i \in I} \eta_i$. If
		\begin{equation}\label{equ:key}
			\left(1 + \frac{1}{\eta}\right) \sum_{i \in I} a_i |A_i \setminus A_i^+| < |V(H)|,
		\end{equation}
		then there exists a non-empty induced subgraph $H' \subseteq H$ such that
		\begin{itemize}
			\item [(a)] $d_{H'}(v) \geq a_i$ for all $v \in V(H') \cap A_i$ and $i\in I$;
			\item [(b)] $|V(H)\backslash V(H')| \leq \sum_{i \in I} a_i |A_i \setminus V(H')| \leq \left(1 + \frac{1}{\eta}\right) \sum_{i \in I} a_i |A_i \setminus A_i^+|$, where $V(H)\backslash V(H')\subseteq \bigcup_{i \in I} A_i$.
		\end{itemize}
	\end{lemma}

	\begin{proof}
		We prove this using the following greedy algorithm.
		Initially, let $H_0 = H$. Suppose that $H_k$ is well-defined for some $k \geq 0$.
		If there exist $i \in I$ and a vertex $v_k \in V(H_k) \cap A_i$ such that $d_{H_k}(v_k) < a_i$,
		then set $H_{k+1} = H_k - \{v_k\}$ and proceed to the next iteration; otherwise, terminate and output $H_k$ as $H'$.
		We aim to show that this $H'$ is the desired induced subgraph.

		Assume that this algorithm terminates at $H'=H_t$.
		Let $U=V(H)\backslash V(H')$. So $U=\{v_0,v_1,...,v_{t-1}\}$.
		We now estimate the number $T$ of edges deleted in this process.
		It is clear from the algorithm that
		\begin{align*}
			T=\sum_{k=0}^{t-1} d_{H_k}(v_k)<\sum_{i\in I} a_i |U\cap A_i|=\sum_{i\in I} a_i\cdot (|U\cap A^+_i|+|U\cap (A_i\backslash A^+_i)|).
		\end{align*}
		On the other hand, we see that $T=e_H(U)+e_H(U,V(H'))$ is at least
		\begin{align*}
			\frac{1}{2}\sum_{v\in U} d_H(v)
			\ge \frac{1}{2}\sum_{i\in I} \sum_{v\in U\cap A^+_i} d_H(v)\ge \sum_{i\in I}(1+\eta_i)a_i\cdot|U\cap A^+_i|.
		\end{align*}
		Combining the above two inequalities, we have
		$$\sum_{i\in I}(1+\eta_i)a_i\cdot|U\cap A^+_i|\leq T\leq \sum_{i\in I} a_i\cdot(|U\cap A^+_i|+|U\cap (A_i\backslash A^+_i)|),$$
		implying that $$\sum_{i\in I} \eta_ia_i\cdot |U\cap A^+_i|\leq \sum_{i\in I}a_i\cdot |U\cap (A_i\backslash A^+_i)|.$$
		Using this inequality, we obtain that
		\begin{align*}
			\sum_{i\in I} a_i\cdot |U\cap A_i|&= \sum_{i\in I} a_i\cdot |U\cap A^+_i|+\sum_{i\in I} a_i\cdot |U\cap (A_i\backslash A^+_i)|\\
			&\leq \frac{1}{\eta}\cdot \sum_{i\in I} \eta_i a_i\cdot |U\cap A^+_i|+\sum_{i\in I} a_i\cdot |U\cap (A_i\backslash A^+_i)|\\
			&\leq (1+\frac{1}{\eta})\cdot \sum_{i\in I} a_i |U\cap (A_i\backslash A^+_i)|\le (1+\frac{1}{\eta})\cdot \sum_{i\in I} a_i |A_i\backslash A^+_i|.
		\end{align*}
		Note that $U\subseteq \cup_{i\in I} A_i$ and each $a_i$ is a positive integer,
		so together with \eqref{equ:key} we have
		$$|V(H)\backslash V(H')|=|U|\leq \sum_{i\in I} a_i\cdot |U\cap A_i|=\sum_{i\in I} a_i\cdot |A_i \setminus V(H')|\leq (1+\frac{1}{\eta})\cdot \sum_{i\in I} a_i |A_i\backslash A^+_i|<|V(H)|,$$ establishing item (b).
		This also shows that the algorithm must terminate at a non-empty subgraph $H'$.
		It is clear from the algorithm that item (a) holds.
	\end{proof}

	\section{Finding bisections with internal degree constraints}\label{sec:int}
	\noindent In this section, we prove our results on graph partitions with internal degree constraints.
	These are obtained from Theorem~\ref{Thm:min-indegree}, a general result applicable to all graphs that establishes the existence of certain tripartitions with degree constraints.
	From a probabilistic perspective, Theorem~\ref{Thm:min-indegree} states that for any $c \in (0,1)$, every $n$-vertex graph $G$ admits a tripartition $V(G) = A \cup B \cup C$ with approximately prescribed part sizes $\left(\frac{1-c}{2}\right)n$, $\left(\frac{1-c}{2}\right)n$, and $cn$, respectively,
	such that every vertex in $A$ has at least roughly half of its expected neighbors within $A$,
	every vertex in $B$ has at least roughly half of its expected neighbors within $B$,
	and every vertex in $C$ has approximately the expected number of neighbors in both $A$ and $B$.
	By appropriately redistributing the vertices of $C$ between $A$ and $B$, we can then derive the corresponding bisection results.

	Before stating the central result of this section, we introduce some necessary definitions.
	Throughout this section, let $c\in [0,1)$ and $\varepsilon\in (0,1)$ be fixed constants.
	Let $d_\varepsilon:=d_{\varepsilon,1}$ be the positive constant obtained from Lemma~\ref{Lem:exist C} with respect to $\varepsilon$.
	For every non-negative integer $i$, we define
	\begin{equation}\label{equ:phi}
		\varphi_{c,\varepsilon}(i)=\left(\frac{1-c}{4}\right)i-\left(2d_\varepsilon\cdot i^{\frac{1}{2}(1+\varepsilon)}+\varepsilon\cdot i\right).
	\end{equation}
	Since $\frac{1}{2}(1+\varepsilon)<1$, we have $\varphi_{c,\varepsilon}(i)=\left(\frac{1-c}{4}-\varepsilon\right)i-o(i)$ as $i$ tends to infinity.
	By treating $\varepsilon$ as an arbitrarily small positive constant, we can express $\varphi_{c,\varepsilon}(i) \to \left(\frac{1-c}{4}\right)i - o(i)$ as $\varepsilon \to 0$.
	
	\begin{theorem}\label{Thm:min-indegree}
		Let $c, \varepsilon$ be constants with $0\leq c<1$ and $0<\varepsilon\leq \frac{1-c}{4}.$
		Then there exists $n_0(c,\varepsilon)$ such that every graph $G$ with $n\geq n_0(c,\varepsilon)$ vertices admits a tripartition $V(G)=A\cup B\cup C$ such that
		\begin{itemize}
			\item[(a).] $\left(\frac{1-c-3\varepsilon}{2}\right)n\leq |A|,|B|\leq \left(\frac{1-c-\varepsilon}{2}\right)n$, and consequently $(c+\varepsilon)n\leq |C|\leq \left(c+3\varepsilon\right)n$,
			\item[(b).] $d_{A}(x)\geq \varphi_{c,\varepsilon}(d_G(x))$ for every vertex $x\in A$,
			\item[(c).] $d_{B}(y)\geq \varphi_{c,\varepsilon}(d_G(y))$ for every vertex $y\in B$, and
			\item[(d).] $d_{A}(z)\geq 2\cdot \varphi_{c,\varepsilon}(d_G(z))$ and $d_B(z)\geq 2\cdot \varphi_{c,\varepsilon}(d_G(z))$ for every vertex $z\in C$
		\end{itemize}
	\end{theorem}

	The remainder of this section is organized as follows.
	In Subsection~\ref{subsec:int-bisec-exact},
	we derive the internal degree results (including Theorems~\ref{Thm:int-bisec} and~\ref{Thm:int-bisec-exact}, and Corollaries~\ref{coro:int-bisec-main} and \ref{coro:int-bisec}) as consequences of Theorem~\ref{Thm:min-indegree}.
	Subsection~\ref{subsec:pf-int-bisec} provides the complete proof of Theorem~\ref{Thm:min-indegree}.

	\subsection{Proof of Theorems~\ref{Thm:int-bisec} and~\ref{Thm:int-bisec-exact}, and Corollaries~\ref{coro:int-bisec-main} and \ref{coro:int-bisec} via Theorem~\ref{Thm:min-indegree}}\label{subsec:int-bisec-exact}
	\noindent
	Assuming Theorem~\ref{Thm:min-indegree},
	we now derive our internal degree results.
	First, we observe that Theorem~\ref{Thm:int-bisec} follows asymptotically from the special case $c=0$ of Theorem~\ref{Thm:min-indegree}.
	
	\begin{proof}[\bf Proof of Theorem~\ref{Thm:int-bisec}.]
		It suffices to prove the following statement: for arbitrarily small $\varepsilon>0$,
		there exists an integer $i_0:=i_0(\varepsilon)$ such that every graph $G$ has a bisection, where every vertex $v$ with $d_G(v)\geq i_0$ has at least $\left(\frac{1}{4}-\varepsilon\right)d_G(v)$ neighbors in its own part.
		
		Fix $c=0$, and let $n_0(0,\varepsilon/2)$ be the constant from Theorem~\ref{Thm:min-indegree} (with parameters $c=0$ and $\varepsilon/2$).
		Consider any graph $G$ with at least $i_0$ vertices, where $i_0\geq n_0(0,\varepsilon/2)$ is sufficiently large.
		Let $V^i$ denote the set of vertices with degree $i$ in $G$.
		From \eqref{equ:phi}, we see that for all $i\geq i_0$, it holds
		$$\varphi_{0,\varepsilon/2}(i)=\frac{i}{4}-\left(2d_{\varepsilon/2}\cdot i^{\frac{1}{2}(1+\varepsilon/2)}+(\varepsilon/2)\cdot i\right)\geq \left(\frac{1}{4}-\varepsilon\right)i.$$
		Applying Theorem~\ref{Thm:min-indegree} with parameters $c=0$ and $\varepsilon/2$, we obtain a tripartition $V(G) = A \cup B \cup C$ satisfying $|A|\leq \frac{n}{2}$, $|B| \leq \frac{n}{2}$,
		every vertex $x\in (A\cup B)\cap V^i$ has at least $\varphi_{0,\varepsilon/2}(i)$ neighbors in its own part,
		and moreover, every vertex $z\in C\cap V^i$ has at least $2\varphi_{0,\varepsilon/2}(i)$ neighbors in each of $A$ and $B$.
		By distributing vertices of $C$ between $A$ and $B$ appropriately, we can form a bisection $A'\cup B'$ of $G$
		such that every vertex $v$ with $d_G(v)\geq i_0$ has at least $\varphi_{0,\varepsilon/2}(d_G(v))\geq \left(\frac{1}{4}-\varepsilon\right)d_G(v)$ neighbors in its own part,
		completing the proof.
	\end{proof}
	
	The proof of Theorem~\ref{Thm:int-bisec-exact} is a direct consequence of Theorem~\ref{Thm:min-indegree},
	which provides the exact degree bounds for the tripartition construction.

	\begin{proof}[\bf Proof of Theorem~\ref{Thm:int-bisec-exact}.]
		Fix constants $0\leq c<1$ and $0<\varepsilon\leq 1-c.$
		Set $\varepsilon'=(1-c)^2\varepsilon/40$.
		Let $G$ be a graph with $\delta(G) \geq \left(\frac{4}{1-c}+\varepsilon\right)k$, where $k \geq n_0(c,\varepsilon')$ is sufficiently large.
		Applying Theorem~\ref{Thm:min-indegree} with parameters $c$ and $\varepsilon'$ gives a tripartition $V(G) = A \cup B \cup C$ satisfying that
		$$\left(\frac{1-c-\varepsilon}{2}\right)n\leq \left(\frac{1-c-3\varepsilon'}{2}\right)n\leq |A|,|B|\leq \left(\frac{1-c-\varepsilon'}{2}\right)n\leq \left(\frac{1-c}{2}\right)n,$$
		every vertex $x\in A\cup B$ with degree $i$ in $G$ has at least
		$$\varphi_{c,\varepsilon'}(i)\geq \left(\frac{1-c}{4}-2\varepsilon'\right)i\geq \left(\frac{1-c}{4}-2\varepsilon'\right)\cdot \left(\frac{4}{1-c}+\varepsilon\right)k=\left(1+\bigg(\frac{1-c}{20}-2\varepsilon'\bigg)\varepsilon\right)\cdot k\geq k$$\footnote{Here we use the facts that $i\geq \delta(G)\geq \left(\frac{4}{1-c}+\varepsilon\right)k$ and $\varepsilon'=(1-c)^2\varepsilon/40$.}neighbors in its own part,
		and moreover, every vertex $z\in C$ with degree $i$ in $G$ has at least $2\varphi_{c,\varepsilon'}(i)\geq 2k$ neighbors in both sets $A$ and $B$.
		This completes the proof.
	\end{proof}

	Now we can derive Corollaries~\ref{coro:int-bisec-main} and \ref{coro:int-bisec} from Theorem~\ref{Thm:int-bisec-exact} promptly.
	
	\begin{proof}[\bf Proof of Corollaries~\ref{coro:int-bisec-main} and \ref{coro:int-bisec}.]
		First we consider Corollary~\ref{coro:int-bisec-main}.
		Fix $\varepsilon\in (0,1)$.
		Set $c=1-\varepsilon$ and let $k_1(\varepsilon)=k_0(c,\varepsilon)$ be the constant obtained from Theorem~\ref{Thm:int-bisec-exact}.
		It suffices to prove the statement of Corollary~\ref{coro:int-bisec-main} for sufficiently large integers $k\geq k_1(\varepsilon)$.
		Let $G$ be a graph with minimum degree at least $c_\varepsilon\cdot k$, where $c_\varepsilon:=\frac{4}{\varepsilon}+\varepsilon$.
		Using Theorem~\ref{Thm:int-bisec-exact}, there exists a tripartition $V(G)=A\cup B\cup C$ such that
		$|A|,|B|\leq \frac{\varepsilon n}{2}$, $|C|\geq (1-\varepsilon)n$,
		every vertex in $A\cup B$ has at least $k$ neighbors in its own part,
		and every vertex in $C$ has at least $2k$ neighbors in both $A$ and $B$.
		Distributing the vertices of $C$ arbitrarily between $A$ and $B$ could yield a bisection $A' \cup B'$
		such that every vertex has at least $k$ neighbors in its own part,
		and every vertex from $C$ has at least $k$ neighbors in the opposite part.
		Since $|C|\geq (1-\varepsilon)n$,
		at least $(1-\varepsilon)n$ vertices satisfy condition (2),
		completing the proof of Corollary~\ref{coro:int-bisec-main}.
		
		For Corollary~\ref{coro:int-bisec},
		its first assertion follows easily from the $c=0$ case of Theorem~\ref{Thm:int-bisec-exact}.
		To see the ``moreover'' part, we apply Theorem~\ref{Thm:int-bisec-exact} with $c=\frac14$.
		Let $k\geq k_0(\frac14,\varepsilon)$ and let $G$ be a graph with minimum degree at least $\left(\frac{16}{3} + \varepsilon\right)k$.
		Theorem~\ref{Thm:int-bisec-exact} implies the existence of a tripartition $V(G)=A\cup B\cup C$ such that $|A|,|B|\leq \frac{3n}{8}$,
		$|C|\geq \frac{n}{4}$, every vertex in $A\cup B$ has at least $k$ neighbors in its own part,
		and every vertex in $C$ has at least $2k$ neighbors in both $A$ and $B$.
		Partition $C=C_1\cup C_2$, with $|C_1|=\lfloor\frac{n}{2}\rfloor-|A|$.
		Let $A'=A\cup C_1$ and $B'=B\cup C_2$.
		Then $A'\cup B'$ forms a bisection of $G$.
		It is evident that both $G[A']$ and $G[B']$ have minimum degree at least $k$,
		and the bipartite subgraph $(A',B')_G$ has at least $2k\cdot |C|\geq \frac{kn}{2}$ edges, i.e., it has average degree at least $k$.
		Hence, $A'\cup B'$ is the desired bisection.
	\end{proof}

	\subsection{Proof of Theorem~\ref{Thm:min-indegree}}\label{subsec:pf-int-bisec}
	\noindent Throughout this subsection, we adopt the following notation.
	Fix constants $c\in [0,1)$ and $\varepsilon\in (0,\frac{1-c}{4}]$.
	Let $d_\varepsilon:=d_{\varepsilon,1}$ be the constant obtained from Lemma~\ref{Lem:exist C} (when applying it with $\varepsilon$).
	We can rewrite the expression for $\varphi_{c,\varepsilon}(i)$ from \eqref{equ:phi} as follows:
	\begin{equation}\label{equ:phi2}
		\varphi(i):=\varphi_{c,\varepsilon}(i)=\left(\frac{1-c}{4}-\frac{\mu_i}{2}\right)\cdot i, \mbox{~~where~~} \mu_i=4d_\varepsilon\cdot i^{\frac{1}{2}(\varepsilon-1)}+2\varepsilon.
	\end{equation}
	For brevity, we will treat $\varphi(i)$ as an integer throughout this subsection.\footnote{Observant readers will notice that rounding $\varphi(i)$ down to its floor will only affect the degree of each vertex by at most one, making this adjustment negligible for our results.}
	Let $G$ be an $n$-vertex graph, where $n\geq n_0(c,\varepsilon)$ is sufficiently large.
	Let $V^i$ represent the set of all vertices of degree $i$ in $G$.
	
	For a subset $S \subseteq V(G)$, we say that a vertex $v\in V^i$ is {\bf $S$-good} if
	\begin{equation}\label{equ:goodness}
		d_S(v)\ge 2(1+\mu_i)\varphi(i).
	\end{equation}
	We denote the set of all $S$-good vertices in $V^i$ by $V^i_S$.
	Let $I$ denote the set of all integers $i$ with $\varphi(i)>0$, and call the vertices in $\bigcup_{i\in I}V^i$ {\it active}.
	For Theorem~\ref{Thm:min-indegree}, it suffices to consider active vertices.
	This is because $\varphi(j) \leq 0$ for each $j\notin I$, and the degree condition in Theorem~\ref{Thm:min-indegree} trivially holds for vertices in $V^j$.
	
	We aim to establish the existence of a tripartition $V(G) = A \cup B \cup C$ that satisfies conditions (a)-(d) stated above.
	The construction of this tripartition unfolds in two major stages:
	\begin{itemize}
		\item An initial random partitioning step, which provides some essential yet lasting properties, and
		\item A deterministic refinement step, involving multiple rounds of local vertex relocations.
	\end{itemize}
	These two stages are addressed separately in the following two lemmas.
	
	We now prove the first-stage lemma, which is obtained by modifying a random partition.
	Its property (1c), established through Lemma~\ref{Lem:exist C}, serves as the most important foundation for the subsequent refinements.
	In a nutshell, all upcoming vertex relocation operations and their effects can be effectively controlled by the specific quantity defined in property (1c).

	\begin{lemma}\label{Lem:indegree-tripar1}
		The graph $G$ has a tripartition $V(G)=A\cup B\cup C$ such that the following hold:
		\begin{itemize}
			\item [(1a).] $\left(\frac{1-c}{2}-\frac{11}{10}\varepsilon\right)n\leq |A|,|B|\leq \left(\frac{1-c}{2}-\frac{9}{10}\varepsilon\right)n$,
			\item [(1b).]  every vertex in $C$ is both $A$-good and $B$-good, and
			\item [(1c).] $\sum_{i\ge 1} i\cdot |V^i\backslash (V^i_A\cap V^i_B)|\leq \frac{\varepsilon^2n}{10^4}.$
		\end{itemize}
	\end{lemma}
	
	\begin{proof}
		Consider a random tripartition $V(G)=A\cup B\cup C$, where each vertex is independently placed in sets $A, B$ and $C$
		with probabilities $\frac{1-c}{2}-\varepsilon,  \frac{1-c}{2}-\varepsilon$ and $c+2\varepsilon$, respectively.
		
		Note that both $|A|$ and $|B|$ are distributed as $X=\Bin(n,\frac{1-c}{2}-\varepsilon)$.
		Thus by Lemma~\ref{Lem:Chernoff} we have
		\begin{align}
			\label{equ:chern1}
			\Prob\left(|X-\E[X]|>\frac{\varepsilon}{10}\E[X]\right)\leq 2e^{-\Omega(\varepsilon^2n)}< \frac14,
		\end{align}
		where the last inequality holds because $n\geq n_0(c,\varepsilon)$ is sufficiently large.
		Let $$\lambda_i=\frac{4d_\varepsilon}{1-c}\cdot i^{\frac{1}{2}(\varepsilon-1)}.$$
		Then we see that
		\begin{align}\label{equ:int-bise-mean}
			\left(\frac{1-c}{2}-\varepsilon\right)\left(1-\lambda_i\right)i>\left(\frac{1-c}{2}-\frac{(1-c)\lambda_i}{2}-\varepsilon\right)i=\left(\frac{1-c}{2}-\frac{\mu_i}{2}\right) i> 2(1+\mu_i)\varphi(i),
		\end{align}
		where the last inequality follows by \eqref{equ:phi2} that $\varphi(i)=\left(\frac{1-c}{4}-\frac{\mu_i}{2}\right)i$.
		Recall the definition that $V^i_A\cap V^i_B$ denotes the set of vertices $v\in V^i$ satisfying $d_A(v)\ge 2(1+\mu_i)\varphi(i)$ and $d_B(v)\ge 2(1+\mu_i)\varphi(i)$.
		For every $v\in V^i$, both $d_A(v)$ and $d_B(v)$ are distributed as $Y=\Bin(i,\frac{1-c}{2}-\varepsilon)$.
		Together with \eqref{equ:int-bise-mean} and Lemma~\ref{Lem:Chernoff}, this implies that the probability $p_i=\Prob\left(v\in V^i\backslash (V^i_A\cap V^i_B)\right)$ satisfies
		\begin{align}\notag
			p_i&\leq 2\cdot\Prob\left(Y< 2(1+\mu_i)\varphi(i)\right)\le 2\cdot \Prob\left(Y<\left(\frac{1-c}{2}-\varepsilon\right)\left(1-\lambda_i\right)i\right)=2\cdot\Prob(Y< (1-\lambda_i)\E[Y])\\
			&\leq 4\cdot \exp(-\lambda_i^2\E[Y]/3)=4\cdot \exp\left(-\frac{16d_\varepsilon^2}{3(1-c)^2}\left(\frac{1-c}{2}-\varepsilon\right)i^{\varepsilon}\right)\le 4\cdot \exp(-d_\varepsilon^2i^{\varepsilon}),\label{equ:pi}
		\end{align}
		where the last inequality holds because $\varepsilon\leq \frac{1-c}{4}$ and thus $\frac{16d_\varepsilon^2}{3(1-c)^2}\left(\frac{1-c}{2}-\varepsilon\right)\geq \frac{4d_\varepsilon^2}{3(1-c)}\geq d_\varepsilon^2$.
		Let $$Z=\sum_{i\geq 1} i\cdot |V^i\backslash (V^i_A\cap V^i_B)|.$$
		Using \eqref{equ:pi} and Lemma~\ref{Lem:exist C}, we obtain that
		\begin{align}\notag
			\E[Z]&=\sum_{i\geq 1}\sum_{v\in V^i}i\cdot \Prob(v\in V^i\backslash (V^i_A\cap V^i_B))=\sum_{i\geq 1}i\cdot p_i\cdot |V^i|\\
			&\leq \sum_{i\geq 1} i\cdot 4\exp(-d_\varepsilon^2i^{\varepsilon})\cdot|V^i|\leq
			\left(\sum_{i=1}^{+\infty}i\cdot \exp(-d_\varepsilon^2i^{\varepsilon})\right)\cdot 4n\le \frac{\varepsilon^2n}{25000}.\label{equ:EZ}
		\end{align}
		By Markov's inequality (Lemma \ref{Lem:Markov}), we have
		\begin{equation}\label{equ:mark1}
			\Prob\left(Z\ge 2\cdot \E[Z]\right)\le \frac{1}{2}.
		\end{equation}
		Combining \eqref{equ:chern1}, \eqref{equ:EZ}, and \eqref{equ:mark1}, we conclude that there exists a tripartition $V(G)=A\cup B\cup C$ satisfying
		\begin{equation}\label{equ:ABsize}
			\left(1-\frac{\varepsilon}{10}\right)\left(\frac{1-c}{2}-\varepsilon\right)n\leq |A|,|B|\leq \left(1+\frac{\varepsilon}{10}\right)\left(\frac{1-c}{2}-\varepsilon\right)n
		\end{equation}
		and
		\begin{equation}
			\label{equ:weightsum1}
			\sum_{i\geq 1} i\cdot |V^i\backslash (V^i_A\cap V^i_B)|=Z<2\cdot \E[Z]\leq \frac{\varepsilon^2n}{10^4}.
		\end{equation}
		
		Next, we move all vertices in $C\cap\left(\cup_{i\ge 1} V^i\backslash (V^i_A\cap V^i_B)\right)$ to $A\cup B$,
		distributing these vertices in an arbitrary manner. We denote the resulting tripartition as $V(G)=A'\cup B'\cup C'$.
		We now claim $A'\cup B'\cup C'$ is the desired tripartition.
		It is important to observe that since $A\subseteq A'$, any vertex which is $A$-good is automatically $A'$-good.
		Thus we have $V^i_A\subseteq V^i_{A'}$ and similarly, $V^i_B\subseteq V^i_{B'}$.
		By \eqref{equ:weightsum1}, this implies that
		$$ \sum_{i\geq 1} i\cdot |V^i\backslash (V^i_{A'}\cap V^i_{B'})|\leq \sum_{i\geq 1} i\cdot |V^i\backslash (V^i_A\cap V^i_B)|\leq \frac{\varepsilon^2n}{10^4},$$
		establishing property (1c) for the tripartition $A'\cup B'\cup C'$.
		Also, as there is no vertex of $V^i\backslash (V^i_A\cap V^i_B)$ belonging to $C'$, we obtain that $C'\cap V^i\subseteq V^i_A\cap V^i_B\subseteq V^i_{A'}\cap V^i_{B'}$,
		i.e., property (1b) holds for $A'\cup B'\cup C'$.
		Finally, using \eqref{equ:weightsum1} again,
		the number of vertices moved out of $C$ is $$|C\cap\left(\cup_{i\ge 1} V^i\backslash (V^i_A\cap V^i_B)\right)|\leq \sum_{i\geq 1} i\cdot |V^i\backslash (V^i_A\cap V^i_B)|<\frac{\varepsilon^2n}{10^4}.$$
		So evidently, $0\le |A'|-|A|\le \varepsilon^2n/10^4$ and $0\le |B'|-|B|\le \varepsilon^2n/10^4$.
		By \eqref{equ:ABsize}, we can derive that $$\left(\frac{1-c}{2}-\frac{21}{20}\varepsilon\right)n< |A|,|B|<\left(\frac{1-c}{2}-\frac{19}{20}\varepsilon\right)n.$$
		Putting everything together, we see that property (1a) holds for the parts $A'$ and $B'$.
		This finishes the proof of Lemma~\ref{Lem:indegree-tripar1}.
	\end{proof}
	
	Next, we present our key technical lemma (Lemma~\ref{Lem:indegree-tripar2}) in this section.
	At a high level, it builds upon the tripartition $A \cup B \cup C$ from Lemma~\ref{Lem:indegree-tripar1} to construct a modified tripartition $A' \cup B' \cup C'$ that both preserves properties (1a)-(1c) and ensures all vertices in $A'$ satisfy condition (b) of Theorem~\ref{Thm:min-indegree}, or equivalently property (2d) of this lemma.
	To complete the proof of Theorem~\ref{Thm:min-indegree}, we will apply this lemma a second time to ensure that all vertices in a modified version of $B'$ also satisfy condition (c) of Theorem~\ref{Thm:min-indegree}, where property (2e) becomes crucial.
	The proof of this lemma is deterministic and proceeds in three concrete steps.
	First, inside $A$ we apply Lemma~\ref{Lem:key} to identify a large robust core $A^*$ whose vertices already satisfy the desired internal-degree condition.
	Second, we run a greedy procedure that repeatedly moves a bad vertex of $A$, together with some of its current neighbors in $C$, into $B$; the weighted estimate from property (1c) of Lemma~\ref{Lem:indegree-tripar1} guarantees that only a very small set of vertices is affected, and the core $A^*$ remains untouched.
	Third, after this greedy cleaning, we repair any remaining deficit inside $A$ by moving a small set of vertices from the surviving part of $C$ back into $A$.
	The formal proof below follows exactly these three steps.
	
	\begin{lemma}
		\label{Lem:indegree-tripar2}
		Let $A\cup B\cup C$ be a tripartition of $V(G)$ such that every vertex in $C$ is both $A$-good and $B$-good, $|A|>\frac{\varepsilon n}{1000}$, and
		\begin{equation}
			\label{equ:indegree condition2}
			\sum_{i\ge 1} i\cdot |(A\cap V^i)\backslash V^i_A|\le \frac{\varepsilon^2n}{250}.
		\end{equation}
		Then there exists a tripartition $A'\cup B'\cup C'$ of $V(G)$ such that the following hold:
		\begin{itemize}
			\item[(2a).] $B\subseteq B'$ and $C'\subseteq C$,\footnote{One cannot establish a containment relationship between $A$ and $A'$ in this lemma.}
			\item[(2b).] $|A|-\frac{\varepsilon n}{500}\le |A'|\le |A|+\frac{\varepsilon n}{500}$, $|B|\leq |B'|\leq |B|+\frac{\varepsilon n}{500}$ and $|C|-\frac{\varepsilon n}{500}\leq |C'|\leq |C|$,
			\item[(2c).] every vertex in $C'$ is both $A'$-good and $B'$-good,
			\item[(2d).] every vertex $x\in A'\cap V^i$ satisfies $d_{A'}(x)\ge \varphi(i)$, and
			\item[(2e).] every vertex in $B'\backslash B$ is $B'$-good. In particular, this shows that for each integer $i\geq 1$,
			\begin{equation}\label{equ:indeg B'}
				(B'\cap V^i)\backslash V^i_{B'}\subseteq (B\cap V^i)\backslash V^i_B.
			\end{equation}
		\end{itemize}
	\end{lemma}
	
	\begin{proof}
		Let $A\cup B\cup C$ be a tripartition of $V(G)$ given by the lemma.
		In what follows, beginning with $A\cup B\cup C$, we will define a sequence of operations that carefully relocate vertices.
		These operations will produce a series of tripartitions of $G$,
		each satisfying progressively stronger properties that bring us closer to the desired properties.
		Importantly, in each round, only a small fraction of vertices are affected.
		This ensures that the final tripartition retains a distribution very close to the original one.
		
		\medskip
		
		\noindent{\it Step 1: extracting a robust core inside $A$.}
		First, we identify a large subset $A^* \subseteq A$ satisfying property (2d).
		Later, we will see that this subset $A^*$ ensures that most vertices in $A$ remain unchanged throughout subsequent operations.
		To construct $A^*$, we apply Lemma~\ref{Lem:key} to $H=G[A]$.
		Let $I$ denote the set of all integers $1\leq i\leq n$ with $\varphi(i)>0$.
		Let $A_i:=A\cap V^i, \eta_i:=\mu_i$ and $a_i:=\varphi(i)$ for each $i\in I$.
		By the definition of $A_i^+$,
		we observe $$A_i^+ =\{v \in A_i : d_H(v) \geq 2(1+\eta_i)a_i\}=\{v\in A\cap V^i: d_A(v)\geq 2(1+\mu_i)\varphi(i)\}= A\cap V_A^i.$$
		Since $\eta=\min \eta_i\geq 2\varepsilon$ and $a_i=\varphi(i)\leq \frac{i}{4}$,
		it holds from \eqref{equ:indegree condition2} that
		$$\Omega:=(1+\frac{1}{\eta}) \sum_{i\in I} a_i \cdot|A_i\backslash A^+_i|\leq
		\frac{1}{\varepsilon} \sum_{i\in I} \frac{i}{4}\cdot |(A\cap V^i)\backslash V^i_A|\leq \frac{\varepsilon n}{1000}<|A|=|V(H)|,$$
		where the final inequality uses the assumption $|A|>\varepsilon n/1000$.
		By Lemma \ref{Lem:key}, there is a subset $A^*$ of $A$ such that
		\begin{equation}\label{equ:indeg deg(A*)}
			d_{A^*}(x)\ge \varphi(i) \text{~~for every } x\in A^*\cap V^i,
			\mbox{~~~and~~~}
			|A\backslash A^*|\le \sum_{i\in I} \varphi(i)\cdot |(A\cap V^i)\backslash A^*|\le \Omega\le \frac{\varepsilon n}{1000}.
		\end{equation}

		\medskip
		
		\noindent{\it Step 2: greedily removing bad vertices from $A$.}
		Next, we implement a greedy algorithm that iteratively moves a small proportion of vertices from $A \cup C$ into $B$.
		The algorithm produces an intermediate tripartition $V(G) = A_1 \cup B_1 \cup C_1$ satisfying $A^* \subseteq A_1$.
		Initially, let $(A^{0}, B^{0}, C^{0}) = (A, B, C)$.
		Given that $A^{k}\cup B^{k}\cup C^{k}$ is defined,
		if there exists a vertex $v_k\in A^{k}\cap V^i$ for some $i\in I$ satisfying
		\begin{equation}
			\label{equ:indeg alg}
			|N_G(v_k)\cap (A^{k}\cup C^{k})|<\varphi(i),
		\end{equation}
		then update $A^{k+1}\cup B^{k+1}\cup C^{k+1}$ by
		$$A^{k+1}=A^k-\{v_k\},~C^{k+1}=C^k\backslash N_G(v_k) \text{~~and~~} B^{k+1}=B^k\cup \{v_k\}\cup (N_G(v_k)\cap C^k);$$
		otherwise, terminate this algorithm and return the current tripartition as $A_1 \cup B_1 \cup C_1$.
		
		We analyze the properties of the tripartition $A_1 \cup B_1 \cup C_1$.
		By construction, for every iteration $k$ we have the following nested containments: $A_1\subseteq A^{k+1}\subseteq A^k\subseteq A$,
		$C_1\subseteq C^{k+1}\subseteq C^k\subseteq C$, and $B\subseteq B^k\subseteq B^{k+1}\subseteq B_1$.
		Moreover, every vertex $x\in A_1\cap V^i$ satisfies
		\begin{equation}
			\label{equ:indeg A'}
			|N_G(x)\cap (A_1\cup C_1)|\ge \varphi(i).
		\end{equation}
		Indeed, no vertex of $A^*$ can ever be selected as some $v_k$: if $x\in A^*\cap V^i$ and $x\in A^k$, then $A^*\subseteq A^k\cup C^k$ and hence
		$|N_G(x)\cap (A^k\cup C^k)|\ge d_{A^*}(x)\ge \varphi(i)$ by \eqref{equ:indeg deg(A*)}, contradicting \eqref{equ:indeg alg}.
		Therefore $A^*\subseteq A_1\subseteq A$, and
		\begin{equation}
			\label{equ:indeg A-A'}
			|A\backslash A_1|\le |A\backslash A^*|\le \frac{\varepsilon n}{1000},
		\end{equation}
		where $A\backslash A_1=\{v_k:k\ge 0\}$.
		Since each $v_k \in V^i$ satisfies $|N_G(v_k) \cap C^k| < \varphi(i)$ (by \eqref{equ:indeg alg}), we bound:
		\begin{equation}
			\label{equ:indeg C-C'}
			|C\backslash C_1|=\sum_{k} |N_G(v_k)\cap C^k|<\sum_{i\in I} \varphi(i)\cdot |(A\backslash A_1)\cap V^i|\le
			\sum_{i\in I} \varphi(i)\cdot |(A\backslash A^*)\cap V^i|\le \frac{\varepsilon n}{1000},
		\end{equation}
		where the last inequality follows from \eqref{equ:indeg deg(A*)}. Therefore,
		\begin{equation}
			\label{equ:indeg B'-B}
			|B_1\backslash B|=|A\backslash A_1|+|C\backslash C_1| \le \frac{\varepsilon n}{500}.
		\end{equation}
		
		We claim that for every $k\geq 0$,
		\begin{equation}
			\label{equ:indeg C^k}
			\text{ each vertex in } C^k \text{ is both } A^k\text{-good  and } B^k\text{-good}.
		\end{equation}
		We prove this by induction on $k$. The base case $k=0$ holds since every vertex in $C$ is $A$-good and $B$-good by assumption.
		Now suppose \eqref{equ:indeg C^k} holds for some $k \geq 0$. Consider any vertex $x \in C^{k+1} = C^k \setminus N_G(v_k)$.
		Since $x\in C^{k+1}\subseteq C^k$, by induction $x$ is $B^k$-good.
		Since $B^k \subseteq B^{k+1}$, this evidently implies that $x$ is also $B^{k+1}$-good.
		Moreover, since $x$ is not adjacent to $v_k$, crucially we have $d_{A^k}(x) = d_{A^{k+1}}(x)$,
		which implies that $x$ remains $A^{k+1}$-good.
		This proves \eqref{equ:indeg C^k} and in particular, implies that
		\begin{equation}
			\label{equ:indeg C'}
			\text{ each vertex in } C_1 \text{ is both } A_1\text{-good  and } B_1\text{-good}.
		\end{equation}
		
		We also claim that for every $k\geq 0$,
		\begin{equation}
			\label{equ:indeg B^k}
			\text{ each vertex in } B^{k+1}\backslash B^k \text{ is } B^k\text{-good}.
		\end{equation}
		Note that $B^{k+1}\backslash B^k =\{v_k\}\cup (N_G(v_k)\cap C^k)$, so by \eqref{equ:indeg C^k}
		it is enough to show that $v_k$ is $B^k$-good, i.e., $|N_{G}(v_k)\cap B^k|\geq 2(1+\mu_i)\varphi(i)$ when $v_k\in V^i$.
		By \eqref{equ:indeg alg}, we have $|N_G(v_k)\cap B^k|=i-|N_G(v_k)\cap (A^k\cup C^k)|> i-\varphi(i)$.
		Using \eqref{equ:phi2}, we compute
		\begin{align*}
			i-\varphi(i)-2(1+\mu_i)\varphi(i)
			&=\left(1-(3+2\mu_i)\left(\frac{1-c}{4}-\frac{\mu_i}{2}\right)\right)i\\
			&=\left(\frac{1+3c}{4}+\left(1+\frac{c}{2}\right)\mu_i+\mu_i^2\right)i\ge 0.
		\end{align*}
		Hence $i-\varphi(i)\ge 2(1+\mu_i)\varphi(i)$, proving \eqref{equ:indeg B^k}.
		Since $B^k$ grows to $B_1$ as $k$ increases, we derive from \eqref{equ:indeg B^k} that
		\begin{equation}
			\label{equ:indeg B-good}
			\text{ each vertex in } B_1\backslash B \text{ is } B_1\text{-good}.
		\end{equation}

		\medskip
		
		\noindent{\it Step 3: repairing the remaining deficits in $A_1$.}
		Let us point out that the tripartition $A_1 \cup B_1 \cup C_1$ satisfies all desired properties of Lemma~\ref{Lem:indegree-tripar2}, except for property (2d).
		To fix this, we move some vertices from $C_1$ into $A_1$ as follows.
		For every $x\in A_1\cap V^i$,
		if $d_{A_1}(x)<\varphi(i)$, then in view of \eqref{equ:indeg A'} there exists a subset $R_x$
		containing exactly $\varphi(i)-d_{A_1}(x)$ neighbors of $x$ in $C_1$;
		otherwise $d_{A_1}(x)\geq \varphi(i)$, let $R_x=\emptyset$.
		Note that if $R_x\neq \emptyset$ for some $x\in A_1$, then $d_{A^*}(x)\le d_{A_1}(x)<\varphi(i)$,
		which together with \eqref{equ:indeg deg(A*)} imply that $x\in A_1\backslash A^*$.
		We then move all vertices in the set
		\begin{equation}\label{equ:indeg R}
			R:=\cup_{x\in A_1} R_x=\cup_{x\in A_1\backslash A^*} R_x
		\end{equation}
		from $C_1$ into $A_1$.
		This results in a new tripartition, denoted by $V(G)=A_2\cup B_2\cup C_2$, where
		$$A_2=A_1\cup R,~ B_2=B_1, \text{ and } C_2=C_1\backslash R.$$
		
		We now show $A_2\cup B_2\cup C_2$ is the desired tripartition for this lemma.
		Clearly, $B\subseteq B_1=B_2$ and $C_2\subseteq C_1\subseteq C$, thus property (2a) holds.
		For every $x\in A_1\cap V^i$, from the construction we see $d_{A_2}(x)\geq \varphi(i)$.
		By \eqref{equ:indeg C'}, every vertex in $R\subseteq C_1$ is $A_1$-good.
		Therefore, every vertex $x\in A_2\cap V^i$ satisfies that $d_{A_2}(x)\ge \varphi(i)$,
		establishing property (2d) for $A_2\cup B_2\cup C_2$.
		Since $A_1\subseteq A_2$ and $B_1=B_2$,
		every $A_1$-good (or $B_1$-good) vertex remains $A_2$-good (or $B_2$-good).
		Therefore \eqref{equ:indeg C'} and \eqref{equ:indeg B-good} imply properties (2c) and (2e) for $A_2\cup B_2\cup C_2$, respectively.
		It remains to show property (2b) for $A_2\cup B_2\cup C_2$.
		To do so, we need to estimate the size of $R$.
		We can derive that
		$$|R|\le \sum_{x\in A_1\backslash A^*} |R_x|\le \sum_{i} \varphi(i)\cdot |(A\cap V^i)\backslash A^*|\le \frac{\varepsilon n}{1000},$$
		where the first inequality follows by \eqref{equ:indeg R}, the second one follows by the fact $A_1\subseteq A$ and by the construction that $|R_x|\leq \varphi(i)$ for each $x\in A_1\cap V^i$,
		and the last one follows by \eqref{equ:indeg deg(A*)}.
		This upper bound on $|R|$, together with \eqref{equ:indeg A-A'}, \eqref{equ:indeg C-C'}, and \eqref{equ:indeg B'-B},
		establishes property (2b) for the tripartition $A_2 \cup B_2 \cup C_2$, thereby completing the proof of Lemma~\ref{Lem:indegree-tripar2}.
	\end{proof}

	We are ready to complete the proof of Theorem~\ref{Thm:min-indegree}.
	
	\medskip
	
	\noindent {\bf Proof of Theorem~\ref{Thm:min-indegree}.}
	Fix constants $c, \varepsilon$ with $0\leq c<1$ and $0<\varepsilon\leq \frac{1-c}{4}.$
	Let $G$ be a graph with $n\geq n_0(c,\varepsilon)$ vertices.
	By Lemma~\ref{Lem:indegree-tripar1}, there exists a tripartition $V(G)=A\cup B\cup C$ such that
	$\left(\frac{1-c}{2}-\frac{11}{10}\varepsilon\right)n\leq |A|,|B|\leq \left(\frac{1-c}{2}-\frac{9}{10}\varepsilon\right)n$,
	every vertex in $C$ is both $A$-good and $B$-good, and
	\begin{equation*}
		\sum_{i\ge 1} i\cdot |V^i\backslash (V^i_A\cap V^i_B)|\leq \frac{\varepsilon^2n}{10^4}.
	\end{equation*}
	Since $A\cap B=\emptyset$, this clearly implies that
	\begin{equation}\label{equ:indegethm1}
		\sum_{i\ge 1} i\cdot |(A\cap V^i)\backslash V^i_A|+\sum_{i\ge 1} i\cdot |(B\cap V^i)\backslash V^i_B|\le \frac{\varepsilon^2n}{10^4}.
	\end{equation}
	Moreover, the size bound from Lemma~\ref{Lem:indegree-tripar1} and the assumption $\varepsilon\leq \frac{1-c}{4}$ imply that
	$$|A|\ge \left(\frac{1-c}{2}-\frac{11}{10}\varepsilon\right)n>\frac{\varepsilon n}{1000}.$$
	So the tripartition $A\cup B\cup C$ satisfies all assumptions of Lemma~\ref{Lem:indegree-tripar2}.
	Using Lemma~\ref{Lem:indegree-tripar2}, there exists a tripartition $V(G)=A'\cup B'\cup C'$ satisfying properties (2a)-(2e).
	In particular, we can derive that
	$\left(\frac{1-c}{2}-\frac{6}{5}\varepsilon\right)n\leq |A'|,|B'|\leq \left(\frac{1-c}{2}-\frac{4}{5}\varepsilon\right)n$,
	every vertex in $C'$ is $A'$-good and $B'$-good, and
	\begin{equation}\label{equ:indegethm2}
		\text{ every } x\in A'\cap V^i \text{ satisfies } d_{A'}(x)\ge \varphi(i).
	\end{equation}
	Moreover, using property (2e) and the above inequality \eqref{equ:indegethm1}, we can obtain that
	\begin{equation*}
		\sum_{i\geq 1} i\cdot |(B'\cap V^i)\backslash V^i_{B'}|\leq \sum_{i\geq 1} i\cdot |(B\cap V^i)\backslash V^i_B|\leq \frac{\varepsilon^2n}{10^4}.
	\end{equation*}
	
	Since $\varepsilon\leq \frac{1-c}{4}$ and $|B'|\ge \left(\frac{1-c}{2}-\frac{4}{5}\varepsilon\right)n$, we also have $|B'|>\frac{\varepsilon n}{1000}$.
	By swapping the roles of $A'$ and $B'$, we see that the tripartition $A'\cup B'\cup C'$ also
	satisfies the assumptions of Lemma \ref{Lem:indegree-tripar2}.
	Applying Lemma \ref{Lem:indegree-tripar2} again (i.e., exchanging the roles of $A'$ and $B'$ here),
	we can find a new tripartition $V(G)=A''\cup B''\cup C''$ such that the following properties hold:
	\begin{itemize}
		\item[(3a).] $A'\subseteq A''$ and $C''\subseteq C'$,
		\item[(3b).] $\left(\frac{1-c-3\varepsilon}{2}\right)n\leq |A''|,|B''|\leq \left(\frac{1-c-\varepsilon}{2}\right)n$,
		\item[(3c).] each vertex in $C''$ is both $A''$-good and $B''$-good,
		\item[(3d).] every vertex $x\in B''\cap V^i$ satisfies $d_{B''}(x)\ge \varphi(i)$, and
		\item[(3e).] every vertex in $A''\backslash A'$ is $A''$-good.
	\end{itemize}
	
	Since $A' \subseteq A''$ (from property (3a)), we can derive from \eqref{equ:indegethm2} and property (3e) that
	\begin{equation*}\label{equ:indegethm6}
		\text{ every vertex } x\in A''\cap V^i \text{ satisfies } d_{A''}(x)\ge \varphi(i).
	\end{equation*}
	Using this, along with properties (3b) through (3d), we conclude that the obtained tripartition $V(G) = A'' \cup B'' \cup C''$ satisfies all conditions of Theorem~\ref{Thm:min-indegree}, thereby completing the proof.
	\qed
	
	\section{Finding bisections with external degree constraints}\label{sec:ext}
	\noindent
	The main result of this section, Theorem~\ref{Thm:min-outdegree}, establishes graph tripartitions satisfying external degree constraints.
	It serves as the natural external-degree analogue of Theorem~\ref{Thm:min-indegree} (which concerns internal degree constraints).
	From a probabilistic perspective, Theorem~\ref{Thm:min-outdegree} asserts that for any $c \in [0,1)$, every $n$-vertex graph $G$ admits a tripartition $V(G) = A \cup B \cup C$ with approximately prescribed part sizes $\left(\frac{1-c}{2}\right)n$, $\left(\frac{1-c}{2}\right)n$, and $cn$, respectively, such that every vertex in $A\cup B$ has at least roughly half of its expected neighbors in the other part, and every vertex in $C$ has approximately the expected number of neighbors in both $A$ and $B$.
	As an immediate consequence, by suitably redistributing vertices from $C$ to $A$ and $B$,
	we obtain corresponding results for bisections with external degree constraints.


	For disjoint subsets $A$ and $B$ of $V(G)$, we note that $(A,B)$ denotes the induced bipartite subgraph of $G$ with parts $A$ and $B$.
	Given constants $c\in [0,1)$ and sufficiently small $\varepsilon\in (0,1)$,
	we introduce the following function (which is slightly different from the one in \eqref{equ:phi}):
	\begin{align}\label{equ:psi}
		\psi_{c,\varepsilon}(i)=\left(\frac{1-c}{4}\right)i-\left(2d_{c,\varepsilon}\cdot i^{\frac{1}{2}(1+\varepsilon)}+\varepsilon\cdot i\right),
	\end{align}
	where $d_{c,\varepsilon}:=d_{\varepsilon,1-c}=1000/(\sqrt{1-c}\cdot \varepsilon^2)$ is the constant supplied by Lemma~\ref{Lem:exist C} with $\rho=1-c$, so that
	\begin{align}\label{equ:sum-i-ext}
		\sum_{i=1}^{+\infty} i\cdot \exp(-d_{c,\varepsilon}^2i^{\varepsilon})\leq \frac{(1-c)\varepsilon^2}{10^5}.
	\end{align}
	Notably, we have $\psi_{c, \varepsilon}(i) \to \left(\frac{1-c}{4}\right)i - o(i)$ as $\varepsilon\to 0$.
	The formal statement is as follows.

	\begin{theorem}\label{Thm:min-outdegree}
		Let $c, \varepsilon$ be constants with $0\leq c<1$ and $0<\varepsilon\leq \frac{1-c}{10}.$
		Then there exists $n_0(c,\varepsilon)$ such that every graph $G$ with $n\geq n_0(c,\varepsilon)$ vertices admits a tripartition $V(G)=A\cup B\cup C$ such that
		\begin{itemize}
			\item[(a).] $\left(\frac{1-c-3\varepsilon}{2}\right)n\leq |A|,|B|\leq \left(\frac{1-c-\varepsilon}{2}\right)n$,
			\item[(b).] $d_{(A,B)}(v)\geq \psi_{c,\varepsilon}(d_G(v))$ for every vertex $v\in A\cup B$,
			\item[(c).] $d_{A}(v)\geq 2\cdot \psi_{c,\varepsilon}(d_G(v))$ and $d_B(v)\geq 2\cdot \psi_{c,\varepsilon}(d_G(v))$ for every vertex $v\in C$
		\end{itemize}
	\end{theorem}
	
	The remainder of this section is organized as follows.
	In Subsection~\ref{subsec:ext-bisec-exact}, we derive Theorems~\ref{Thm:ext-bisec} and~\ref{Thm:ext-bisec-exact}, along with Corollaries~\ref{coro:ext-bisec-main} and~\ref{coro:ext-bisec}, from Theorem~\ref{Thm:min-outdegree}.
	Subsection~\ref{subsec:pf-ext-bisec} then contains the proof of Theorem~\ref{Thm:min-outdegree}.
	
	\subsection{Proofs of Theorems~\ref{Thm:ext-bisec} and~\ref{Thm:ext-bisec-exact}, and Corollaries~\ref{coro:ext-bisec-main} and~\ref{coro:ext-bisec} via Theorem~\ref{Thm:min-outdegree}}\label{subsec:ext-bisec-exact}
	\noindent
	Assuming Theorem~\ref{Thm:min-outdegree}, we now derive the corresponding external-degree results.
	As in Subsection~\ref{subsec:int-bisec-exact}, the proofs are short once the tripartition theorem is available.
	
	\begin{proof}[\bf Proof of Theorem~\ref{Thm:ext-bisec}.]
		It suffices to prove the following statement: for arbitrarily small $\varepsilon>0$,
		there exists an integer $i_0:=i_0(\varepsilon)$ such that every graph $G$ has a bisection where every vertex $v$ with $d_G(v)\geq i_0$ has at least $\left(\frac{1}{4}-\varepsilon\right)d_G(v)$ neighbors in its opposite part.
		
		Fix $c=0$, and let $n_0(0,\varepsilon/2)$ be the constant from Theorem~\ref{Thm:min-outdegree} (with parameters $c=0$ and $\varepsilon/2$).
		Consider any graph $G$ with at least $i_0$ vertices, where $i_0\geq n_0(0,\varepsilon/2)$ is sufficiently large.
		Let $V^i$ denote the set of vertices with degree $i$ in $G$.
		From \eqref{equ:psi}, we see that for all $i\geq i_0$, it holds
		$$\psi_{0,\varepsilon/2}(i)=\frac{i}{4}-\left(2d_{0,\varepsilon/2}\cdot i^{\frac{1}{2}(1+\varepsilon/2)}+(\varepsilon/2)\cdot i\right)\geq \left(\frac{1}{4}-\varepsilon\right)i.$$
		Applying Theorem~\ref{Thm:min-outdegree} with parameters $c=0$ and $\varepsilon/2$, we obtain a tripartition $V(G)=A\cup B\cup C$ satisfying $|A|\leq \frac{n}{2}$, $|B|\leq \frac{n}{2}$,
		every vertex $v\in (A\cup B)\cap V^i$ has at least $\psi_{0,\varepsilon/2}(i)$ neighbors in its opposite part,
		and every vertex $z\in C\cap V^i$ has at least $2\psi_{0,\varepsilon/2}(i)$ neighbors in both $A$ and $B$.
		By distributing the vertices of $C$ between $A$ and $B$ appropriately, we form a bisection $A'\cup B'$ of $G$.
		Every vertex $v$ with $d_G(v)\geq i_0$ then has at least $\psi_{0,\varepsilon/2}(d_G(v))\geq \left(\frac{1}{4}-\varepsilon\right)d_G(v)$ neighbors in its opposite part, completing the proof.
	\end{proof}
	
	\begin{proof}[\bf Proof of Theorem~\ref{Thm:ext-bisec-exact}.]
		Fix constants $0\leq c<1$ and $0<\varepsilon\leq 1-c$.
		Set $\varepsilon'=(1-c)^2\varepsilon/40$.
		Let $G$ be a graph with $\delta(G)\geq \left(\frac{4}{1-c}+\varepsilon\right)k$, where $k\geq n_0(c,\varepsilon')$ is sufficiently large.
		Applying Theorem~\ref{Thm:min-outdegree} with parameters $c$ and $\varepsilon'$ gives a tripartition $V(G)=A\cup B\cup C$ satisfying
		$$\left(\frac{1-c-\varepsilon}{2}\right)n\leq \left(\frac{1-c-3\varepsilon'}{2}\right)n\leq |A|,|B|\leq \left(\frac{1-c-\varepsilon'}{2}\right)n\leq \left(\frac{1-c}{2}\right)n,$$
		every vertex $v\in A\cup B$ with degree $i$ in $G$ has at least
		$$\psi_{c,\varepsilon'}(i)\geq \left(\frac{1-c}{4}-2\varepsilon'\right)i\geq \left(\frac{1-c}{4}-2\varepsilon'\right)\cdot \left(\frac{4}{1-c}+\varepsilon\right)k\geq k$$
		neighbors in its opposite part,
		and every vertex $z\in C$ has at least $2\psi_{c,\varepsilon'}(d_G(z))\geq 2k$ neighbors in both $A$ and $B$.
		This is exactly the required tripartition.
	\end{proof}
	
	\begin{proof}[\bf Proof of Corollaries~\ref{coro:ext-bisec-main} and \ref{coro:ext-bisec}.]
		First we consider Corollary~\ref{coro:ext-bisec-main}.
		Fix $\varepsilon\in (0,1)$.
		Set $c=1-\varepsilon$ and let $k_1(\varepsilon)=k_0(c,\varepsilon)$ be the constant obtained from Theorem~\ref{Thm:ext-bisec-exact}.
		It suffices to prove the statement of Corollary~\ref{coro:ext-bisec-main} for sufficiently large integers $k\geq k_1(\varepsilon)$.
		Let $G$ be a graph with minimum degree at least $c_\varepsilon\cdot k$, where $c_\varepsilon:=\frac{4}{\varepsilon}+\varepsilon$.
		Using Theorem~\ref{Thm:ext-bisec-exact}, there exists a tripartition $V(G)=A\cup B\cup C$ such that
		$|A|,|B|\leq \frac{\varepsilon n}{2}$, $|C|\geq (1-\varepsilon)n$,
		every vertex in $A\cup B$ has at least $k$ neighbors in the opposite part,
		and every vertex in $C$ has at least $2k$ neighbors in both $A$ and $B$.
		By distributing the vertices of $C$ between $A$ and $B$ appropriately, we obtain a bisection $A'\cup B'$
		such that every vertex has at least $k$ neighbors in the opposite part,
		and every vertex from $C$ has at least $k$ neighbors in its own part.
		Since $|C|\geq (1-\varepsilon)n$,
		at least $(1-\varepsilon)n$ vertices satisfy condition (ii), proving Corollary~\ref{coro:ext-bisec-main}.
		
		For Corollary~\ref{coro:ext-bisec}, its first assertion follows easily from Theorem~\ref{Thm:ext-bisec-exact} with $c=0$. 
        To see the ``moreover'' part, we apply Theorem~\ref{Thm:ext-bisec-exact} with $c=\frac14$.
		Let $k\geq k_0(\frac14,\varepsilon)$ and let $G$ be a graph with minimum degree at least $\left(\frac{16}{3} + \varepsilon\right)k$.
		Theorem~\ref{Thm:ext-bisec-exact} implies the existence of a tripartition $V(G)=A\cup B\cup C$ such that $|A|,|B|\leq \frac{3n}{8}$,
		$|C|\geq \frac{n}{4}$, every vertex in $A\cup B$ has at least $k$ neighbors in  the opposite part,
		and every vertex in $C$ has at least $2k$ neighbors in both $A$ and $B$.
		Partition $C=C_1\cup C_2$, with $|C_1|=\lfloor\frac{n}{2}\rfloor-|A|$.
		Let $A'=A\cup C_1$ and $B'=B\cup C_2$.
		Then $A'\cup B'$ forms a bisection of $G$.
		It is evident that $(A',B')_G$ has minimum degree at least $k$.
        Note that both $C_1$ and $C_2$ have size at least $\frac{n}{8}$. 
        So we have $e(A')\geq 2k\cdot |C_1|\geq \frac{kn}{4}$ and thus $A'$ has average degree at least $k$.
        Similarly, $B'$ has average degree at least $k$.
		Hence, $A'\cup B'$ is the desired bisection.
	\end{proof}
	
	The proof of Theorem~\ref{Thm:min-outdegree} follows the same two-stage framework as Theorem~\ref{Thm:min-indegree},
	consisting of an initial random partitioning step and a deterministic refinement step.
	The first step remains nearly identical, though it requires a modified version of Lemma~\ref{Lem:indegree-tripar1}.
	The second step, however, requires substantially different strategies for relocating vertices between parts.
	It is important to emphasize that both proofs rely on the same foundational fact that all vertex relocation operations and their effects are effectively controlled by the quantity defined in property (1c) of Lemma~\ref{Lem:indegree-tripar1} (or Lemma~\ref{Lem:outdegree-tripar}).
	
	\subsection{Proof of Theorem~\ref{Thm:min-outdegree}}\label{subsec:pf-ext-bisec}
	\noindent
	Throughout the rest of this section, we fix constants $c\in [0,1)$ and $\varepsilon\in (0,\frac{1-c}{10}]$.
	We have an equivalent expression for $\psi(i):=\psi_{c,\varepsilon}(i)$ as follows
	\begin{align}\label{equ:psi-2}
		\psi(i)=\left(\frac{1-c}{4}-\frac{\mu_i}{2}\right)\cdot i, \mbox{~~where~~} \mu_i=4d_{c,\varepsilon}\cdot i^{\frac{1}{2}(\varepsilon-1)}+2\varepsilon.
	\end{align}
	Let $G$ be a graph with $n$ vertices, where $n\geq n_0(c,\varepsilon)$ is sufficiently large.
	Let $V^i$ be the set of all vertices of degree $i$ in $G$.
	Let $I$ denote the set of all integers $i$ with $\psi(i)>0$.
	As we observed earlier, it suffices to consider vertices $v\in V^i$ with $i\in I$;
	we say such vertices are {\it active}.
	
	We need a modified version of Lemma~\ref{Lem:indegree-tripar1}.
	A crucial technical requirement arises later in the deterministic refinement step,
	where we need $i = O(\psi(i))$ (see the upcoming \eqref{equ:outdeg 3}), which does not hold for small values of $i$.
	This obstacle necessitates replacing $\psi(i)$ with the following modified function
	\begin{align}\label{equ:phi*}
		\psi^*(i)=\max\left\{\psi(i),~\left(\frac{1-c}{8}\right)i\right\} \mbox{ for all } i\in I.
	\end{align}
	This modification leads to a revised definition of ``$S$-goodness'' for subsets $S$ of $V(G)$, along with corresponding adjustments to Lemma~\ref{Lem:indegree-tripar1}.
	To be precise, for each $i\in I$, we define
	\begin{equation}\label{equ:eta}
		\lambda_i=\frac{4d_{c,\varepsilon}}{1-c}\cdot i^{\frac{1}{2}(\varepsilon-1)}
		\mbox{ ~ and ~ }
		\eta_i=\lt\{\begin{array}{ll}
			\mu_i & \text{ if } \psi(i)\geq \left(\frac{1-c}{8}\right)i \\
			4\varepsilon\lambda_i/(1-c) & \text{ if } \psi(i)< \left(\frac{1-c}{8}\right)i
		\end{array}\rt.
	\end{equation}
	In this section, for a subset $S\subseteq V(G)$,
	we say an active vertex $x\in V^i$ is {\bf $S$-good} if
	\begin{equation}\label{def:outdegree S-good}
		d_S(x)\geq 2(1+\eta_i)\cdot \psi^*(i)=\lt\{\begin{array}{ll}
			2(1+\mu_i)\cdot\psi(i) & \text{ if } \psi(i)\geq \left(\frac{1-c}{8}\right)i \\
			((1-c)/4+\varepsilon\lambda_i)\cdot i & \text{ if } \psi(i)< \left(\frac{1-c}{8}\right)i.
		\end{array}\rt.
	\end{equation}
	For $i\in I$, we denote the set of all $S$-good vertices in $V^i$ by $V^i_S$.
	
	The following lemma presents a modified version of Lemma~\ref{Lem:indegree-tripar1}.
	We provide a sketch of the proof.
	
	\begin{lemma}\label{Lem:outdegree-tripar}
		The graph $G$ has a tripartition $V(G)=X\cup Y\cup Z$ such that the following hold:
		\begin{itemize}
			\item [(a)] $(\frac{1-c}{2}-\frac{11}{10}\varepsilon)n\le |X|,|Y|\le (\frac{1-c}{2}-\frac{9}{10}\varepsilon)n$,
			\item [(b)]  every active vertex in $Z$ is both $X$-good and $Y$-good, and
			\item [(c)] $\sum_{i\in I} i\cdot |V^i\backslash (V^i_X\cap V^i_Y)|\leq \frac{(1-c)\varepsilon^2n}{10^4}.$
		\end{itemize}
	\end{lemma}
	
	\begin{proof}
		Consider a random tripartition $V(G) = X \cup Y \cup Z$, where each vertex is independently placed in $X$, $Y$, and $Z$ with probabilities $\frac{1-c}{2} - \varepsilon$, $\frac{1-c}{2} - \varepsilon$, and $c + 2\varepsilon$, respectively.
		
		Consider the probability $p_i = \Prob\left(v \in V^i \setminus (V^i_X \cap V^i_Y)\right)$.
		It suffices to consider all $i\in I$, for which $\psi(i)>0$ and thus $\mu_i<\frac{1-c}{2}$.
		We claim that $$p_i\leq \exp(-d_{c,\varepsilon}^2i^{\varepsilon}) \mbox{ ~ holds for all~} i\in I.$$
		For $i\in I$ with $\psi(i)\geq \left(\frac{1-c}{8}\right)i$,
		we observe that the $S$-goodness here is identical to that of Section~3,
		hence through the same arguments of \eqref{equ:int-bise-mean} and \eqref{equ:pi},
		we can derive the conclusion that $p_i\leq \exp(-d_{c,\varepsilon}^2i^{\varepsilon})$.
		
		It remains to consider integers $i\in I$ satisfying $\psi(i) < \left(\frac{1-c}{8}\right)i$.
		This case requires separate treatment, which we address below.
		Note that $(1-c)\lambda_i+2\varepsilon=\mu_i< \frac{1-c}{2}$.
		So we have
		\begin{align}\label{equ:ext-bise-mean}
			\left(\frac{1-c}{2}-\varepsilon\right)(1-\lambda_i)=\frac{1-c}{2}-\frac{(1-c)\lambda_i+2\varepsilon}{2}+\varepsilon \lambda_i\geq \frac{1-c}4+\varepsilon \lambda_i.
		\end{align}
		For any $v\in V^i$,
		both $d_X(v)$ and $d_Y(v)$ are distributed as $W=\Bin(i,\frac{1-c}{2}-\varepsilon)$.
		By \eqref{def:outdegree S-good} and \eqref{equ:ext-bise-mean},
		\begin{align}\notag
			p_i&\le 2\cdot\Prob\left(W<((1-c)/4+\varepsilon\lambda_i)\cdot i\right)\le 2\cdot \Prob\left(W< \left(\frac{1-c}{2}-\varepsilon\right)(1-\lambda_i)i\right) \\\notag
			&=2\cdot\Prob(W<(1-\lambda_i)\E[W])\le 4\cdot \exp(-\lambda_i^2\E[W]/3)\le 4\cdot \exp(-d_{c,\varepsilon}^2i^{\varepsilon}),
		\end{align}
		where the last inequality follows by the same proof as in \eqref{equ:pi}.
		This proves the above claim.
		
		Now we consider $$P=\sum_{i\in I} i\cdot |V^i\backslash (V^i_X\cap V^i_Y)|.$$
		Using the above claim and \eqref{equ:sum-i-ext}, we derive that its expectation satisfies
		$$\E[P]\leq \left(\sum_{i\in I} i\cdot \exp(-d_{c,\varepsilon}^2i^{\varepsilon})\right)\cdot 4n\leq \frac{(1-c)\varepsilon^2n}{25000}.$$
		Now, by following the same arguments as in the remainder of the proof of Lemma~\ref{Lem:indegree-tripar1}, we obtain the desired tripartition.
		We emphasize that, for our purposes, it suffices to consider only the active vertices; that is why property (c) sums only over $i \in I$.
	\end{proof}
	
	With the initial tripartition prepared in Lemma~\ref{Lem:outdegree-tripar},
	we now present the proof of Theorem~\ref{Thm:min-outdegree}.
	The deterministic refinement has the following shape.
	We first apply Lemma~\ref{Lem:key} to the bipartite graph induced by the two large parts in order to extract a large bipartite core whose vertices already have the desired external degree.
	The discarded vertices, together with their neighbors in the third part, form a small exceptional set $W_1$.
	Next, we greedily place into the two large parts any vertex of $W_1$ that already sees enough neighbors on the opposite side.
	The vertices that remain in the reservoir form a graph $G[W_2]$ of large minimum degree, and a maximum cut of this graph supplies the missing external neighbors.
	The detailed proof below follows this plan.
	
	\medskip
	
	\noindent {\bf Proof of Theorem \ref{Thm:min-outdegree}.}
	Fix constants $0\leq c<1$ and $0<\varepsilon\leq \frac{1-c}{10}.$
	Let $G$ be a graph with $n\geq n_0(c,\varepsilon)$ vertices.
	Let $V(G)=X\cup Y\cup Z$ be the tripartition from Lemma \ref{Lem:outdegree-tripar} such that
	$(\frac{1-c}{2}-\frac{11}{10}\varepsilon)n\leq |X|,|Y|\leq (\frac{1-c}{2}-\frac{9}{10}\varepsilon)n$,
	every active vertex in $Z$ is both  $X$-good and $Y$-good, and
	\begin{equation}\label{equ:outdeg 2}
		\sum_{i\in I} i\cdot |V^i\backslash (V^i_X\cap V^i_Y)|\leq \frac{(1-c)\varepsilon^2n}{10^4}.
	\end{equation}
	Recall the definition of $\eta_i$ in \eqref{equ:eta}.
	We claim that
	\begin{equation}\label{equ:outdeg delta}
		\eta_i\geq \frac{\varepsilon}{5} \mbox{ for each } i\in I, \mbox{~ and thus ~} \eta=\min_{i\in I} \eta_i\geq \frac{\varepsilon}{5}.
	\end{equation}
	If $\eta_i=\mu_i$, then $\eta_i=4d_{c,\varepsilon}\cdot i^{\frac{1}{2}(\varepsilon-1)}+2\varepsilon\geq 2\varepsilon$.
	Now assume that $\psi(i)< \left(\frac{1-c}{8}\right)i$ and $\eta_i=4\varepsilon\lambda_i/(1-c)$.
	In this case, using \eqref{equ:psi-2} we see that $(1-c)\lambda_i+2\varepsilon=\mu_i>\frac{1-c}{4}$.
	Since $\varepsilon\leq \frac{1-c}{10}$, this implies that $\lambda_i\geq \frac{1}{20}$ and thus
	$\eta_i=4\varepsilon\lambda_i/(1-c)\geq \frac{\varepsilon}{5}$, proving this claim.

	We first extract a large bipartite core from the initial pair $(X,Y)$.
	Let $H=(X,Y)_G$ be the induced bipartite subgraph of $G$ with parts $X$ and $Y$.
	We propose applying Lemma~\ref{Lem:key} to $H$ by choosing $A_i := V^i \cap V(H)$, $a_i := \psi^*(i)$, and $\eta_i$ as defined in \eqref{equ:eta}, for all $i\in I$.
	Let $A_i^+$ be as defined in Lemma~\ref{Lem:key}, that is,
	$$A_i^+= \{v \in A_i : d_H(v) \geq 2(1+\eta_i)a_i\}=\{v \in A_i : d_H(v) \geq 2(1+\eta_i)\psi^*(i)\}.$$
	Using \eqref{def:outdegree S-good}, we derive that
	$$A_i^+=(V^i_X\cap Y)\cup (V^i_Y\cap X)\supseteq V^i_X\cap V^i_Y\cap A_i.$$
	Since every active vertex in $Z$ is both $X$-good and $Y$-good, we see that for every $i\in I$,
	\begin{equation*}
		A_i\backslash A_i^+\subseteq A_i\backslash (V^i_X\cap V^i_Y)=V^i\backslash (V^i_X\cap V^i_Y).
	\end{equation*}
	This, combined with \eqref{equ:outdeg 2} and \eqref{equ:outdeg delta},
	implies that condition \eqref{equ:key} of Lemma~\ref{Lem:key} is satisfied:
	\begin{equation}\label{equ:outdeg sumH}
		\left(1+\frac{1}{\eta}\right)\cdot \sum_{i\in I} a_i\cdot |A_i\backslash A_i^+|\leq \frac{10}{\varepsilon}\cdot
		\sum_{i\in I}i\cdot |V^i\backslash (V^i_X\cap V^i_Y)|\leq \frac{(1-c)\varepsilon n}{1000}<|V(H)|.
	\end{equation}
	Then Lemma~\ref{Lem:key} guarantees the existence of a non-empty bipartite subgraph $H'$ of $H$ with parts $X' \subseteq X$ and $Y' \subseteq Y$, satisfying that
	\begin{equation}
		\label{equ:outdeg d(H')}
		d_{H'}(x)\geq a_i= \psi^*(i) \geq \psi(i) \text{ for every } x\in (X'\cup Y')\cap V^i,
	\end{equation}
	and
	\begin{equation}\label{equ:outdeg-key}
		|V(H\backslash H')|\le \sum_{i\in I} a_i|A_i\backslash V(H')|\le
		\left(1+\frac{1}{\eta}\right)\cdot \sum_{i\in I} a_i |A_i\backslash A_i^+|\leq \frac{(1-c)\varepsilon n}{1000},
	\end{equation}
	where $(X\backslash X')\cup (Y\backslash Y')=V(H\backslash H')\subseteq \cup_{i\in I} A_i$.
	Using the definition \eqref{equ:phi*} of $\psi^*$, we see that
	$$i\leq \frac{8\psi^*(i)}{1-c}=\frac{8a_i}{1-c} \mbox{ for each } i\in I.$$
	This, together with \eqref{equ:outdeg-key}, implies that
	\begin{equation}
		\label{equ:outdeg 3}
		\sum_{i\in I}i\cdot |V(H\backslash H')\cap V^i|\leq \frac{8}{1-c}\cdot\sum_{i\in I}a_i |V(H\backslash H')\cap V^i|=\frac{8}{1-c}\cdot\sum_{i\in I} a_i|A_i\backslash V(H')|\le \frac{\varepsilon n}{125}.
	\end{equation}
	
	The estimate \eqref{equ:outdeg 3} shows that the exceptional set discarded from $H$ is very small even after weighting vertices by their degrees.
	We now absorb these discarded vertices, together with their neighbors in $Z$, into a new reservoir $W_1$; the remaining vertices of $Z$ will still be good with respect to both sides.
	We further define a 4-partition $V(G)=X_1\cup Y_1\cup Z_1\cup W_1$,
	where $$X_1=X', ~~Y_1=Y', ~~W_1=V(H\backslash H')\cup \left(\bigcup_{x\in V(H\backslash H')} (N_G(x)\cap Z)\right), \mbox{ ~~and~ } Z_1=Z\backslash W_1.$$
	Then every $x\in Z_1$ satisfies that $N_G(x)\cap X\subseteq X_1$ and $N_G(x)\cap Y\subseteq Y_1$.
	So the fact that every active vertex in $Z$ is $X$-good and $Y$-good directly implies that
	\begin{equation}\label{equ:outdeg 4}
		\text{every active vertex in } Z_1 \text{ is } X_1\text{-good and } Y_1\text{-good}.
	\end{equation}
	On the other hand, \eqref{equ:outdeg-key} and \eqref{equ:outdeg 3} show that
	\begin{equation}\label{equ:outdeg D'}
		|W_1|\leq |V(H\backslash H')|+\sum_{i\in I} i\cdot|V(H\backslash H')\cap V^i|\le \frac{\varepsilon n}{50}.
	\end{equation}
	
	We claim that for every $i\in I$ and every vertex $x\in W_1\cap V^i$,
	\begin{equation}\label{equ:outdeg 5}
		|N_G(x)\cap (X_1\cup Y_1\cup W_1)|\ge 4\psi(i).
	\end{equation}
	This is clear for every $x\in V(H\backslash H')\cap V^i$, as $|N_G(x)\cap (X_1\cup Y_1\cup W_1)|=d_G(x)=i\geq 4\psi(i)$.
	It remains to consider $x\in (W_1\cap Z)\cap V^i$.
	Since every active vertex in $Z$ is $X$-good and $Y$-good,
	we see that
	$|N_G(x)\cap (X_1\cup Y_1\cup W_1)|\geq d_{X\cup Y}(x)\ge 4(1+\eta_i)a_i\ge 4 a_i\ge 4\psi(i)$. This proves the claim.
	
	Our next goal is to place as many vertices of $W_1$ as possible directly into $X_1$ or $Y_1$.
	If a vertex already has enough neighbors on one side, we move it to the opposite side and remove it from the reservoir.
	We then apply a greedy algorithm, which iteratively moves some vertices of $W_1$ to $X_1\cup Y_1$ as follows.
	Initially, let $X_2:=X_1, Y_2:=Y_1$ and $W_2:=W_1$.
	If there is a vertex $x\in W_2\cap V^i$ with $d_{X_2}(x)\ge \psi(i)$,
	then update $W_2:=W_2\backslash \{x\}$, $X_2:=X_2$ and $Y_2:=Y_2\cup \{x\}$;
	if there is a vertex $x\in W_2\cap V^i$ with $d_{Y_2}(x)\geq \psi(i)$,
	then update $W_2:=W_2\backslash \{x\}$, $X_2:=X_2\cup \{x\}$ and $Y_2:=Y_2$;
	otherwise, this algorithm terminates.
	When terminating, this results in a new 4-partition $V(G)=X_2\cup Y_2\cup Z_2\cup W_2$ such that
	$$X_1\subseteq X_2,~~Y_1\subseteq Y_2,~~Z_2=Z_1,~~W_2\subseteq W_1,$$ and for every $x\in W_2\cap V^i$, $d_{X_2}(x)<\psi(i)$ and $d_{Y_2}(x)<\psi(i)$.
	The last property, together with \eqref{equ:outdeg 5} and the fact that $X_2\cup Y_2\cup W_2=X_1\cup X_2\cup W_1$, implies that for every $x\in W_2\cap V^i,$ it holds
	$$d_{W_2}(x)\ge 2\psi(i).$$
	At this point the only unresolved vertices are those in $W_2$.
	The previous inequality shows that $G[W_2]$ is still dense enough, so a maximum cut of this graph will provide the missing external neighbors for these remaining vertices.
	Let $(W^+,W^-)$ be a max-cut of $G[W_2]$.
	Then by a well-known property, for every $x\in W_2\cap V^i$,
	\begin{equation}\label{equ:outdeg maxcut}
		d_{(W^+,W^-)}(x)\geq \frac{d_{W_2}(x)}{2}\geq \psi(i).
	\end{equation}
	The above greedy algorithm shows that every vertex $x\in (X_2\backslash X_1)\cap V^i$ satisfies $d_{Y_2}(x)\ge \psi(i)$.
	Combining this with the fact $X' = X_1 \subseteq X_2$ and \eqref{equ:outdeg d(H')}, we obtain
	\begin{align}\label{equ:X2Y2}
		d_{Y_2}(v)\ge \psi(i) \mbox{ for every } v\in X_2\cap V^i, \mbox{ and similarly, } d_{X_2}(v)\ge \psi(i) \mbox{ for every } v\in Y_2\cap V^i.
	\end{align}
	
	Now we define a tripartition $V(G)=X_3\cup Y_3\cup Z_3$ satisfying that
	$$X_3=X_2\cup W^+, ~~  Y_3=Y_2\cup W^-, \text{~~and~~} Z_3=Z_2=Z_1.$$
	We show this is the desired tripartition.
	Firstly, by \eqref{equ:outdeg maxcut} and \eqref{equ:X2Y2}, we see clearly that
	$$d_{(X_3,Y_3)}(x)\geq \psi(i) \mbox{ ~holds for every ~} x\in (X_3\cup Y_3)\cap V^i,$$
	i.e., condition (b) of Theorem~\ref{Thm:min-outdegree} holds.
	Secondly, since $X\backslash W_1\subseteq X_1\subseteq X_3\subseteq X_1\cup W_1\subseteq X\cup W_1$,
	we can derive from \eqref{equ:outdeg D'} that
	$$\left(\frac{1-c-3\varepsilon}{2}\right)n\leq |X|-|W_1|\leq |X_3|\leq |X|+|W_1|\le \left(\frac{1-c-\varepsilon}{2}\right)n;$$
	similarly, we have $\left(\frac{1-c-3\varepsilon}{2}\right)n\leq |Y_3|\leq \left(\frac{1-c-\varepsilon}{2}\right)n$.
	So condition (a) of Theorem~\ref{Thm:min-outdegree} holds.
	Finally, since $X_1\subseteq X_3, Y_1\subseteq Y_3$ and $Z_3=Z_1$,
	by \eqref{equ:outdeg 4} we can derive that every active vertex in $Z_3$ is both $X_3$-good and $Y_3$-good.
	Using \eqref{equ:phi*} and \eqref{def:outdegree S-good},
	this shows that every active vertex $v\in Z_3\cap V^i$ satisfies both $d_{X_3}(v)\geq 2\cdot \psi(i)$ and $d_{Y_3}(v)\geq 2\cdot \psi(i)$,
	so condition (c) of Theorem~\ref{Thm:min-outdegree} holds,
	finishing the proof.
	\qed

	\section{Concluding remarks and open problems}\label{sec:conclusion}
	\noindent
	In this paper, we combine both probabilistic and deterministic arguments to prove several results on bisections in graphs under degree constraints. Our main results in general settings demonstrate that every graph $G$ has a bisection where each vertex $v$ has at least $d_G(v)/4 - o(d_G(v))$ neighbors in its own part, as well as a bisection where each vertex $v$ has at least $d_G(v)/4 - o(d_G(v))$ neighbors in the opposite part.
	It would be extremely interesting to improve the aforementioned $1/4$ constants to $1/2$, as this would be optimal when considering random bisections.
	If this improvement holds true, it would also yield a weak version of the conjecture by Bollob\'as and Scott (Conjecture~8 in \cite{BS}) regarding bisections with external degree constraints.
	
	Specifically, we would like to raise the following minimum-degree version of the question:
	
	\begin{question}\label{que:2k}
		Is there a function $f(k) = 2k + o(k)$ such that every graph with minimum degree at least $f(k)$ has a bisection where each vertex has at least $k$ neighbors in its own part, as well as a bisection where each vertex has at least $k$ neighbors in the opposite part?
	\end{question}
	
	\noindent We demonstrate in Corollaries \ref{coro:int-bisec} and \ref{coro:ext-bisec} that both bounds of $4k + o(k)$ are sufficient. However, Lemma~\ref{Lem:key} presents a significant barrier, preventing us from improving these bounds beyond $4k + o(k)$ in either direction.
	Moreover, there are additional obstacles in our arguments. For the external degree direction,
	the calculation just before \eqref{equ:indeg B-good} indicates that the best possible bound (aside from Lemma~\ref{Lem:key}) that could be derived from our methods is at least $3k + o(k)$.
	In the internal degree direction, the analysis between \eqref{equ:outdeg 5} and \eqref{equ:outdeg maxcut} suggests that our proof cannot achieve a minimum degree bound lower than $4k + o(k)$ on a second occasion.
	To address Question~\ref{que:2k}, we believe that some novel ideas will be necessary.
	
	We note that both Corollaries~\ref{coro:int-bisec-main} and \ref{coro:int-bisec} extend the result \eqref{equ:KO2} of K\"uhn and Osthus \cite{KO}, in addition to the bisection requirement.
	One may wonder if there is a bisection analog for another result \eqref{equ:KO} of K\"uhn and Osthus \cite{KO}.
	The answer is negative even for $k=1$, as demonstrated by the following graphs first provided in \cite{KO}:
	for any $\ell \geq 2$ and sufficiently large $n$, let $G$ be the bipartite graph with parts $X = [n]$ and $Y = \{v_F : F \in \binom{[n]}{\ell}\}$, where $i \in X$ and $v_F \in Y$ are adjacent if and only if $i \in F$.
	Note that the minimum degree of $G$ is $\ell$, which can be arbitrarily large.
	We now show that $G$ has no bisection $A \cup B$ satisfying \eqref{equ:KO} even for $k=1$.
	Suppose for a contradiction that $A \cup B$ is such a bisection, meaning every vertex has at least one neighbor in its own part and every vertex in $A$ has at least one neighbor in $B$.
	Since $|Y|$ is much larger than $|X|$, we have $|A \cap Y| \geq |Y|/3 = \frac{1}{3} \binom{n}{\ell}$.
	If $|A \cap X| < \ell$, then $e(G[A]) < \ell \binom{n-1}{\ell-1} < \frac{1}{3} \binom{n}{\ell} \leq |A \cap Y|$, a contradiction.
	Thus, $|A \cap X| \geq \ell$. Let $F$ be a subset of size $\ell$ in $A \cap X$.
	We observe that the vertex $v_F$ has none of its neighbors in $B$. This leads to a final contradiction, regardless of which part $v_F$ belongs to.

	Our proofs can be extended to derive analogous results for multi-partitions with arbitrary part-size constraints. For internal degrees, based on the arguments in Section~\ref{sec:int}, we can prove the following:
	\begin{theorem}\label{Thm:int-multipart}
		For any $r \geq 2$ and given $\alpha_1, \ldots, \alpha_r \in (0,1)$ with $\sum_{i=1}^r \alpha_i = 1$, every $n$-vertex graph $G$ has an $r$-partition $V(G) = V_1 \cup \ldots \cup V_r$ with $|V_i| = \alpha_i n$ such that every vertex $v\in V_i$ has at least $\left(\frac{\alpha_i}{2} + o(1)\right) d_G(v)$ neighbors in $V_i$.
	\end{theorem}
	\noindent For external degrees, the analogous result is as follows:
	\begin{theorem}\label{Thm:ext-remark}
		For any $r \geq 2$ and given $\alpha_1, \ldots, \alpha_r \in (0,1)$ with $\sum_{i=1}^r \alpha_i = 1$, every $n$-vertex graph $G$ has an $r$-partition $V(G) = V_1 \cup \ldots \cup V_r$ with $|V_i| = \alpha_i n$ such that every vertex $v \in V_i$ satisfies $d_{V(G) \setminus V_i}(v) \geq \frac{1}{2}(1 - \alpha_i + o(1)) \cdot d_G(v)$.
	\end{theorem}
	\noindent Brief proof sketches of Theorems~\ref{Thm:int-multipart} and \ref{Thm:ext-remark} are given in Appendix~\ref{Sec:App-multipart}.
	In particular, the proof of Theorem~\ref{Thm:ext-remark} uses a property regarding a biased Max-$r$-Cut (see Lemma~\ref{Lem:max-r-cut} in Appendix~\ref{Sec:App}), which is used in place of \eqref{equ:outdeg maxcut}.
	In the case $r = 2$, this is equivalent to stating that for any constants $\alpha_1, \alpha_2 \in (0,1)$ with $\alpha_1 + \alpha_2 = 1$,
	every $n$-vertex graph $G$ has a bipartition $V(G) = V_1 \cup V_2$ with $|V_1| = \alpha_1 n$ and $|V_2| = \alpha_2 n$ such that for any $1\leq i\neq j\leq 2$,
	every vertex $v\in V_i$ has at least $\frac{1}{2}(\alpha_j + o(1)) \cdot d_G(v)$ neighbors in $V_j$.
	It would be interesting to see if, for $r \geq 3$, one can strengthen the above result by requiring that for each $v \in V_i$ and for each $j \neq i$, we have $d_{V_j}(v) \geq \frac{1}{2}(\alpha_j + o(1)) \cdot d_G(v)$.
	
	Thomassen \cite{T83} and Hajnal \cite{Haj} proved an analogue of the result \eqref{equ:Thom} for connectivity.
	That is, they established that for all integers $k\geq 1$, there exists a function $g(k)$ such that
	every $g(k)$-connected graph $G$ has a bipartition $V(G)=A\cup B$, where both $G[A]$ and $G[B]$ are $k$-connected.
	We wonder if this result can be extended to bisections.
	
	\begin{question}\label{que:connectivity}
		Is there a function $h(k)$ such that every $h(k)$-connected graph $G$ has a bisection $V(G)=A\cup B$, where both $G[A]$ and $G[B]$ are $k$-connected?
	\end{question}

	\section*{Acknowledgements}
	\noindent The authors are indebted to Chang Shi and Zhiheng Zheng for their careful readings of a preliminary draft and for providing valuable comments that greatly improved the presentation.

	\appendix
	
	\section{A lemma on biased Max-$r$-Cut in graphs}\label{Sec:App}
	
	\begin{lemma}\label{Lem:max-r-cut}
		Given any $r \geq 2$ and $\alpha_1, \ldots, \alpha_r \in (0,1)$ with $\sum_{i=1}^r \alpha_i = 1$,
		every graph $G$ has an $r$-partition $V(G) = U_1 \cup U_2 \cup \ldots \cup U_r$ such that for every $i \neq j$ and every $x \in U_i$, it holds that
		$$\frac{d_{U_i}(x)}{\alpha_i}\le \frac{d_{U_j}(x)}{\alpha_j}.$$
		In particular, this implies that $d_{V(G)\backslash U_i}(x)\ge (1-\alpha_i)\cdot d_G(x)$.
	\end{lemma}
	\begin{proof}
		Consider a nontrivial $r$-partition $V(G)=U_1\cup ...\cup U_r$, which minimizes the function $$f(U_1,...,U_r)=\sum_{i=1}^r \frac{1}{\alpha_i}\cdot e(U_i).$$
		Without loss of generality, consider a vertex $x\in U_1$.
		Let $d_j=d_{U_j}(x)$ for each $j$ and let $d=d_G(x)$.
		If $U_1=\{x\}$, then clearly $d_1/\alpha_1=0\le d_j/\alpha_j$.
		So we have $|U_1|\ge 2$.
		Then $f(U_1,...,U_r)\leq f(U_1\backslash \{x\},...,U_j\cup \{x\},...,U_r)$ for every $j\neq 1$,
		implying that $$\frac{e(U_1)}{\alpha_1}+\frac{e(U_j)}{\alpha_j}\le \frac{e(U_1\backslash \{x\})}{\alpha_1}+\frac{e(U_j\cup \{x\})}{\alpha_j}, \mbox{ and thus ~ } \frac{d_1}{\alpha_1}\leq \frac{d_j}{\alpha_j},$$
		as desired.
		That also is $\alpha_j d_1\le \alpha_1 d_j$ for all $j\in [r]$.
		Summing over all $j$, we have $d_1\leq \alpha_1 d$.
		This proves that $d_{U_i}(x)\le \alpha_i\cdot d_G(x)$ for every $i$ and every $x\in U_i$, finishing the proof.
	\end{proof}
	
	\section{Proof sketches for the multipartition extensions}\label{Sec:App-multipart}
	
	\begin{proof}[\bf Proof sketch of Theorem~\ref{Thm:int-multipart}]
		Fix $\varepsilon>0$ and a small constant $\xi>0$.
		For each $i\in [r]$, let $p_i=(1-\xi)\alpha_i$, and let $p_0=\xi$ be the probability of a reserve part $C$.
		Write $\alpha_*:=\min_{i\in [r]} p_i$ and choose
		$$d_{\boldsymbol{\alpha},\varepsilon}:=d_{\varepsilon,\alpha_*}=\frac{1000}{\sqrt{\alpha_*}\cdot \varepsilon^2}.$$
		For each integer $t\ge 0$ and each $i\in [r]$, define
		\[
		\phi_i(t)=\frac{p_i}{2}\,t-\left(2d_{\boldsymbol{\alpha},\varepsilon}\cdot t^{\frac{1+\varepsilon}{2}}+\varepsilon t\right).
		\]
		Now repeat the proof of Theorem~\ref{Thm:min-indegree} with an $(r+1)$-partition
		\[V(G)=X_1\cup \cdots \cup X_r\cup C\]
		in place of the tripartition $A\cup B\cup C$.
		The random step is the same as in Lemma~\ref{Lem:indegree-tripar1}, except that each vertex is placed in $X_i$ with probability $p_i$ and in $C$ with probability $p_0$.
		The deterministic refinement is then applied successively to the $r$ main parts, exactly as Lemma~\ref{Lem:indegree-tripar2} is applied twice in Section~\ref{sec:int}.
		This yields an $(r+1)$-partition with
		\[
		|X_i|\le \alpha_i n,\qquad \alpha_i n-|X_i|=O(\xi n)\ \text{ for each } i\in [r],
		\]
		such that every vertex $x\in X_i$ satisfies $d_{X_i}(x)\ge \phi_i(d_G(x))$, while every vertex $z\in C$ satisfies $d_{X_i}(z)\ge 2\phi_i(d_G(z))$ for all $i\in [r]$.
		
		Now set $m_i=\alpha_i n-|X_i|$.
		Since $\sum_{i=1}^r m_i=|C|$, we may partition $C$ as
		\[C=C_1\cup \cdots \cup C_r\qquad\text{with}\qquad |C_i|=m_i\ \text{ for each } i\in [r].\]
		Finally define $V_i=X_i\cup C_i$.
		For every $x\in X_i$, we have $d_{V_i}(x)\ge d_{X_i}(x)\ge \phi_i(d_G(x))$.
		For every $z\in C_i$, we likewise have
		\[d_{V_i}(z)\ge d_{X_i}(z)\ge 2\phi_i(d_G(z))\ge \phi_i(d_G(z)).\]
		Because
		\[
		\phi_i(t)=\left(\frac{(1-\xi)\alpha_i}{2}-\varepsilon\right)t-o(t)=\left(\frac{\alpha_i}{2}-O(\xi)-o(1)\right)t,
		\]
		every vertex $v\in V_i$ has at least $\left(\frac{\alpha_i}{2}-O(\xi)-o(1)\right)d_G(v)$ neighbors in $V_i$.
		Letting $\xi\to 0$ gives Theorem~\ref{Thm:int-multipart}.
	\end{proof}
	
	\begin{proof}[\bf Proof sketch of Theorem~\ref{Thm:ext-remark}]
		Fix $\varepsilon>0$ and a small constant $\xi>0$.
		For each $i\in [r]$, let $p_i=(1-\xi)\alpha_i$, let $p_0=\xi$, and set
		\[q_i=(1-\xi)(1-\alpha_i)=\sum_{j\neq i} p_j.\]
		Write $q_*:=\min_{i\in [r]} q_i$ and choose
		$$d'_{\boldsymbol{\alpha},\varepsilon}:=d_{\varepsilon,q_*}=\frac{1000}{\sqrt{q_*}\cdot \varepsilon^2}.$$
		For each integer $t\ge 0$ and each $i\in [r]$, define
		\[
		\psi_i(t)=\frac{q_i}{2}\,t-\left(2d'_{\boldsymbol{\alpha},\varepsilon}\cdot t^{\frac{1+\varepsilon}{2}}+\varepsilon t\right).
		\]
		We now repeat the proof of Theorem~\ref{Thm:min-outdegree} with an $(r+1)$-partition
		\[V(G)=X_1\cup \cdots \cup X_r\cup C,\]
		where each vertex is placed in $X_i$ with probability $p_i$ and in $C$ with probability $p_0$.
		The random step and the analogue of Lemma~\ref{Lem:key} are identical to those in Section~\ref{sec:ext}, so we omit them.
		Exactly as in the proof of Theorem~\ref{Thm:min-outdegree}, after the greedy relocation stage we arrive at a small leftover set $W_2$ together with $r$ main parts $X_1,\ldots,X_r$ such that
		\[
		|X_i|\le \alpha_i n,\qquad \alpha_i n-|X_i|=O(\xi n)\ \text{ for each } i\in [r],
		\]
		every vertex $x\in X_i$ already satisfies
		\[d_{\bigcup_{j\neq i}X_j}(x)\ge \psi_i(d_G(x)),\]
		and every remaining reserve vertex $z\in C$ satisfies
		\[d_{\bigcup_{j\neq i}X_j}(z)\ge 2\psi_i(d_G(z))\qquad\text{for every } i\in [r].\]
		
		The only new ingredient is the final treatment of $W_2$.
		Instead of taking an ordinary maximum cut of $G[W_2]$, we apply Lemma~\ref{Lem:max-r-cut} with the weights $\alpha_1,\ldots,\alpha_r$.
		This gives an $r$-partition
		\[W_2=W_1^+\cup \cdots \cup W_r^+\]
		such that every $x\in W_i^+$ satisfies
		\[
		d_{W_2\setminus W_i^+}(x)\ge (1-\alpha_i)d_{W_2}(x).
		\]
		The same calculation used in the passage from \eqref{equ:outdeg 5} to \eqref{equ:outdeg maxcut} shows that every $x\in W_2$ satisfies
		\[
		d_{W_2}(x)\ge \frac{d_G(x)}{2}-o(d_G(x)).
		\]
		Hence, for every $x\in W_i^+$,
		\[
		d_{W_2\setminus W_i^+}(x)\ge (1-\alpha_i)\left(\frac{d_G(x)}{2}-o(d_G(x))\right)
		=\left(\frac{1-\alpha_i}{2}-o(1)\right)d_G(x)\ge \psi_i(d_G(x)).
		\]
		Now add $W_i^+$ to the $i$-th main part, and then distribute the remaining reserve vertices from $C$ so as to obtain exact sizes $|V_i|=\alpha_i n$ for all $i$.
		Vertices already in the main parts only gain external neighbors, and every reserve vertex moved to $V_i$ still has at least $2\psi_i(d_G(v))\ge \psi_i(d_G(v))$ neighbors outside $V_i$.
		Since
		\[
		\psi_i(t)=\left(\frac{(1-\xi)(1-\alpha_i)}{2}-\varepsilon\right)t-o(t)
		=\left(\frac{1-\alpha_i}{2}-O(\xi)-o(1)\right)t,
		\]
		we obtain
		\[
		d_{V(G)\setminus V_i}(v)\ge \left(\frac{1-\alpha_i}{2}-O(\xi)-o(1)\right)d_G(v)
		\]
		for every $v\in V_i$.
		Letting $\xi\to 0$ proves Theorem~\ref{Thm:ext-remark}.
	\end{proof}

\begin{aicauthors}
\begin{authorinfo}[jie]
  Jie Ma\\
  School of Mathematical Sciences, University of Science and Technology of China\\
  Hefei, Anhui 230026, China\\
  Yau Mathematical Sciences Center, Tsinghua University\\
  Beijing 100084, China\\
  jiema\imageat{}ustc\imagedot{}edu\imagedot{}cn
\end{authorinfo}
\begin{authorinfo}[hehui]
  Hehui Wu\\
  Shanghai Center for Mathematical Sciences, Fudan University\\
  Shanghai 200438, China\\
  hhwu\imageat{}fudan\imagedot{}edu\imagedot{}cn
\end{authorinfo}
\end{aicauthors}

\end{document}